\documentclass[12pt, onecolumn]{IEEEtran}
\usepackage{amsmath,amssymb,euscript,yfonts,psfrag,latexsym,dsfont,graphicx}
\usepackage{bbm,color,amstext,wasysym,subfig,flushend,parskip,balance}
\graphicspath{{./},{./figures/}}

\newtheorem{thm}{Theorem}

\newtheorem{proof}{Proof}

\newcommand{\mR}{{\mathbb R}}

\newcommand{\mE}{{\mathbb E}}

\newcommand{\tr}{\operatorname{trace}}

\newcommand{\argmin}{\operatorname{argmin}}

\title{Optimal transport for Gaussian mixture models}

\author{Yongxin Chen, Tryphon T. Georgiou and Allen Tannenbaum
\thanks{Y.\ Chen is with the Department of Electrical and Computer Engineering, Iowa State University, IA; email: yongchen@iastate.edu}
\thanks{T.\ T. Georgiou is with the Department of Mechanical and Aerospace Engineering, University of California, Irvine, CA; email: tryphon@uci.edu}
\thanks{A.\ Tannenbaum is with the Departments of Computer Science and Applied Mathematics \& Statistics, Stony Brook University, NY; email: allen.tannenbaum@stonybrook.edu}}

\begin{document}
% \nipsfinalcopy is no longer used

\maketitle

\begin{abstract}
We present an optimal mass transport framework on the space of Gaussian mixture models. These models are widely used in statistical inference.  We treat such models as elements on the submanifold of probability densities with Gaussian mixture structure and embed them to the density spaces equipped with Wasserstein metric. An equivalent view relates to discrete measures on the space of Gaussian densities. Our method leads to a natural way to compare, interpolate and average Gaussian mixture models, with low computational complexity. Different aspects of our framework are discussed and several examples are presented for illustration. The method represents a first attempt to study optimal transport problems for probability densities with specific structure that can be suitably exploited.
\end{abstract}

\section{Introduction}
A mixture model is a probabilistic model describing properties of populations with subpopulations. Formally, it is a mixture distribution with each component representing a subpopulation.  Mixture models are widely used in statistics in detecting subgroups, inferring properties of subpopulations, and many other areas \cite{MclPel04}. An important case of mixture models is the so-called Gaussian mixture model (GMM), which is simply a weighted average of several Gaussian distributions. Each Gaussian component stands for a subpopulation. The Gaussian mixture model is commonly used in applications due to its mathematical simplicity as well as efficient algorithms in inference (e.g., Expectation Maximization algorithm).

Optimal mass transport (OMT) is an old but very active research area dealing with probability densities. Starting from the original formulation of Monge \cite{Mon81}, and the relaxation of Kantorovich \cite{Kan42}, and following up with a sequence of important works \cite{Bre91,GanMcc96,Mcc97,JorKinOtt98,BenBre00,OttVil00}, the subject of OMT has become a powerful tool in mathematics, physics, economics, engineering and biology etc \cite{Eva99,HakTanAng04,MueKarKolTan13,CheGeoPav14e,CheGeoPav15b,Gal16,Che16}. Benefitting from the development of alternative algorithms \cite{AngHakTan03,Cut13,BenFroObe14,TabTri14,HabHor15,BenCarCut15,CheGeoPav15a}, OMT has recently found applications in data science \cite{MonMulCut16,ArjChiBot17}. Briefly, OMT deals with problems of transporting masses from an initial distribution to a terminal distribution in a mass preserving manner with minimum cost. When the unit cost is the square of the Euclidean distance, the OMT problem induces an extremely rich geometry for probability densities. It endows a Riemannian metric on the space of probability densities \cite{Otto,Vil03,Vil08}. This geometry enables us to compare, interpolate and average probability densities in a very natural way, which is in line with the needs in a range of applications.

Despite the inherent elegance and beauty OMT geometry,
%in order to apply OMT, one needs to consider the problems on the whole manifold of probability densities.
transport on the entire manifold of probability densities is computationally expensive.
However, in many applications, the probability densities often have specific structure and may be parameterized \cite{AmaNag07}. Thus, we are motivated to study OMT on certain submanifolds of probability densities. To retain the nice properties of OMT, herein, we seek an explicit OMT framework on Gaussian mixture models. The extension to more general structured densities will be a future research topic.

This work is also partially motivated by problems in data science. Meaningful data are often of high dimension and always have some structure. Thus, they are not densely distributed in the high dimensional space. Instead, they usually live in a low dimensional sub-manifold. Besides, in many cases, the data are sparsely distributed among subgroups. The difference between data within a subgroup is way less significant than that between subgroups. For example, the differences between two dogs of the same species are expected to be smaller than that of different species. In such applications, mixture models are suitable, and therefore it is of importance to develop a mathematical framework that respects such data structure.

\section{Background on OMT}
We now give a very brief overview of OMT theory. We only cover materials that are related to the present work. We refer the reader to \cite{Vil03} for more details.

Consider two measures $\mu_0, \mu_1$ on ${\mathbb R}^n$ with equal total mass. Without loss of generality, we take $\mu_0$ and $\mu_1$ to be probability distributions.
In the original formulation of OMT, a transport map
\[
T\;:\;{\mathbb R}^n\to{\mathbb R}^n\;:\;x\mapsto T(x)
\]
is sought that specifies where mass $\mu_0(dx)$ at $x$ should be transported so as to match the final distribution in the sense that $T_\sharp \mu_0=\mu_1$, i.e.
$\mu_1$ is the ``push-forward'' of $\mu_0$ under $T$, meaning
\[
\mu_1(B)=\mu_0(T^{-1}(B))
\]
for every Borel set $B$ in $\mR^n$. Moreover, the map should achieve a minimum cost of transportation
\[
\int_{\mR^n} c(x,T(x))\mu_0(dx).
\]
Here, $c(x,y)$ represents the transportation cost per unit mass from point $x$ to $y$. In this paper we focus on the case when  $c(x,y)=\|x-y\|^2$. To ensure finite cost, it is standard to assume that $\mu_0$ and $\mu_1$ live in the space of probability densities with finite second moments, denoted by $P_2(\mR^n)$.

The dependence of the transportation cost on $T$ is highly nonlinear and a minimum may not exist in general. This fact complicated early analyses of the problem \cite{Vil03}. To circumvent this difficulty, Kantorovich presented a relaxed formulation in 1942. In this, instead of seeking a transport map, one seeks a joint distribution $\Pi(\mu_0,\mu_1)$ on $\mR^n\times\mR^n$, referred to as ``coupling" of $\mu_0$ and $\mu_1$, so that the marginals along the two coordinate directions coincide with $\mu_0$ and $\mu_1$, respectively.
Thus, in the Kantorovich formulation, we solve
    \begin{equation}\label{eq:OptTrans}
        \inf_{\pi\in\Pi(\mu_0,\mu_1)}\int_{\mR^n\times\mR^n}\|x-y\|^2\pi(dxdy).
    \end{equation}

For the case where $\mu_0,\mu_1$ are absolutely continuous with corresponding densities $\rho_0$ and $\rho_1$, it is a standard result that OMT (\ref{eq:OptTrans}) has a unique solution \cite{Bre91,Vil03,Vil08}. Moreover, the unique optimal transport $T$ is the gradient of a convex function $\phi$, i.e.,
	\begin{equation}\label{eq:optimalmap}
		y=T(x)=\nabla\phi(x).
	\end{equation}
%By virtue of the fact that the push-forward of $\mu_0$ under $\nabla\phi$ is $\mu_1$, this function satisfies a particular case of the  Monge-Amp\`ere equation \cite[p.126]{Vil03}, \cite[p.377]{BenBre00}, namely, $\det(H\phi(x))\rho_1(\nabla \phi(x))=\rho_0(x)$, where $H\phi$ is the Hessian matrix of $\phi$, which is a fully nonlinear second-order elliptic equation.
Having the optimal mass transport map $T$, as in (\ref{eq:optimalmap}), the optimal coupling is
	\[
		\pi=({\rm Id}\times T)_\sharp \mu_0,
	\]
where ${\rm Id}$ stands for the identity map. The square root of the minimum of the cost defines a Riemannian metric on $P_2(\mR^n)$, known as the Wasserstein metric $W_2$ \cite{JorKinOtt98,Otto,Vil03,Vil08}. On this Riemannian-type manifold, the geodesic curve connecting $\mu_0$ and $\mu_1$ is given by
	\begin{equation}\label{eq:displacementinterp1}
		\mu_t=(T_t)_\sharp \mu_0,~~~T_t(x)=(1-t)x+tT(x),
	\end{equation}
which is called displacement interpolation. It satisfies
	\begin{equation}\label{eq:W2geodesic}
		W_2(\mu_s,\mu_t) = (t-s) W_2(\mu_0,\mu_1),\quad 0\le s< t\le 1.
	\end{equation}

\subsection{Gaussian marginal distributions}\label{sec:Gaussian}
When both of the marginals $\mu_0, \mu_1$ are Gaussian distributions, the problem can be greatly simplified \cite{Tak11}. In fact, a closed-form solution exists. Denote the mean and covariance of $\mu_i, i=0,1$ by $m_i$ and $\Sigma_i$, respectively. Let $X, Y$ be two Gaussian random vectors associated with $\mu_0, \mu_1$, respectively. Then the cost in (\ref{eq:OptTrans}) becomes
	\begin{equation}\label{eq:expectcost}
		\mE\{\|X-Y\|^2\} = \mE \{\|\tilde X-\tilde Y\|^2\} +\|m_0-m_1\|^2,
	\end{equation}
where $\tilde X = X-m_0, \tilde Y = Y-m_1$ are zero mean versions of $X$ and $Y$. We minimize \eqref{eq:expectcost} over all the possible Gaussian joint distributions between $X$ and $Y$. This gives
	\begin{equation}
	\min_S \left\{\|m_0-m_1\|^2+\tr(\Sigma_0+\Sigma_1-2S) ~\mid~
	\left[\begin{matrix}
	\Sigma_0 & S \\ S^T & \Sigma_1
	\end{matrix}\right]\ge 0\right\},
	\end{equation}
with $S=\mE \{\tilde X \tilde Y^T\}$. The constraint is semidefinite constraint, so the above problem is a semidefinite programming (SDP). It turns out that the minimum is achieved by the unique minimizer in closed-form
	\[
		S= \Sigma_0^{1/2}(\Sigma_0^{1/2}\Sigma_1\Sigma_0^{1/2})^{1/2}\Sigma_0^{-1/2}
	\]
with minimum value
	\[
		W_2(\mu_0,\mu_1)^2 = \|m_0-m_1\|^2+\tr(\Sigma_0+
		\Sigma_1-2(\Sigma_0^{1/2}\Sigma_1\Sigma_0^{1/2})^{1/2}).
	\]
The consequent displacement interpolation $\mu_t$ is a Gaussian distribution with mean $m_t = (1-t)m_0+tm_1$ and covariance
	\begin{equation}\label{eq:disinterpG}
		\Sigma_t = \Sigma_0^{-1/2} \left((1-t)\Sigma_0+t(\Sigma_0^{1/2}\Sigma_1\Sigma_0^{1/2})^{1/2}
		 \right)^2 \Sigma_0^{-1/2}.
	\end{equation}
	
The Wasserstein distance can be extended to singular Gaussian distributions by replacing the inverse by the pseudoinverse $\dagger$, which leads to
	\begin{equation}\label{eq:W2gaussian}
		W_2(\mu_0,\mu_1)^2 = \|m_0-m_1\|^2+\tr(\Sigma_0+
		\Sigma_1-2\Sigma_0^{1/2}((\Sigma_0^{1/2})^\dagger\Sigma_1(\Sigma_0^{1/2})^\dagger)^{1/2}\Sigma_0^{1/2}).
	\end{equation}
In particular, when $\Sigma_0=\Sigma_1=0$, we have
	\[
		W_2(\mu_0,\mu_1) = \|m_0-m_1\|,
	\]
implying that the Wasserstein space of Gaussian distributions, denoted by $G(\mR^n)$, is an extension, at least formally, of the Euclidean space $\mR^n$. 

\section{OMT for Gaussian mixture models}\label{sec:Gaussianmix}
A Gaussian mixture model is an important instance of mixture models, which are commonly used to study properties of populations with several subgroups. Mathematically, a Gaussian mixture model is a probability density consisting of several Gaussian components. Namely, it has the form
	\[
		\mu = p^1 \nu^1 + p^2 \nu^2 + \cdots+p^N \nu^N,
	\]
where each $\nu^k$ is a Gaussian distribution and $p=(p^1,p^2,\ldots,p^N)^T$ is a probability vector. Here the finite number $N$ stands for the number of components of $\mu$. We denote the space of Gaussian mixture distributions by $M(\mR^n)$.

As we have already seen in Section~\ref{sec:Gaussian}, the displacement interpolation of two Gaussian distributions remains Gaussian. This invariance, however, no longer holds for Gaussian mixtures. Yet, the mixture models may contain some physical or statistical features that we may want to retain. This gives rise to the following question we would like to address. How do we establish a geometry that inherits the nice properties of OMT and in the meantime keeps the Gaussian mixture structure?

Our approach relies on a different way of looking at Gaussian mixture models. Instead of treating the given mixture as a distribution on the Euclidean space $\mR^n$, we view it as a discrete distribution on the Wasserstein space of Gaussian distributions $G(\mR^n)$. A Gaussian mixture distribution is equivalent to a discrete measure, and therefore we can apply OMT theory to such discrete measures. We will see next that this strategy retains the Gaussian mixture structure.

Let $\mu_0, \mu_1$ be two Gaussian mixture models of the form
	\[
		\mu_i = p_i^1 \nu_i^1 + p_i^2 \nu_i^2 + \cdots+p_i^{N_i} \nu_i^{N_i},~~i=0, 1.
	\]
Here $N_0$ maybe different to $N_1$. The distribution $\mu_i$ is equivalent to a discrete measure $p_i$ with supports $\nu_i^1, \nu_i^2,\ldots,\nu_i^{N_i}$ for each $i=0,1$. Our framework is built on the discrete OMT problem
	\begin{equation}\label{eq:OMTdiscrete}
	\min_{\pi \in \Pi(p_0,p_1)} \sum_{i,j} c(i,j)\pi(i,j)
	\end{equation}
for these two discrete measures.
Here $\Pi(p_0,p_1)$ denote the space of joint distributions between $p_0$ and $p_1$. The cost $c(i,j)$ is taken to be the square of the Wasserstein metric on $G(\mR^n)$, that is,
	\[
		c(i,j) = W_2(\nu_0^i,\nu_1^j)^2.
	\]
By standard linear programming theory, the discrete OMT problem (\ref{eq:OMTdiscrete}) always has at least one solution. Let $\pi^*$ be a minimizer, and define
	\begin{equation}\label{eq:metric}
		d(\mu_0,\mu_1) = \sqrt{\sum_{i,j} c(i,j)\pi^*(i,j)}.
	\end{equation}
%Then it can be shown that $d(\cdot,\cdot)$ defines a metric on $M(\mR^n)$. The space $M(\mR^n)$ equipped with this metric $d$ is intrinsic in the sense that
%	\[
%		d(\mu_0,\mu_1) = \sup_{0=t_0<t_1<\cdots<t_s=1} \sum_{k} d(\mu_{t_k},\mu_{t_{k+1}}).
%	\]
\begin{thm}\label{thm:metric}
$d(\cdot,\cdot)$ defines a metric on $M(\mR^n)$.
\end{thm}
\begin{proof}
Apparently, $d(\mu_0,\mu_1)\ge 0$ for any $\mu_0,\mu_1\in M(\mR^n)$ and $d(\mu_0,\mu_1)=0$ if and only if $\mu_0=\mu_1$. We next prove the triangular inequality, namely,
	\[
		d(\mu_0,\mu_1) +d(\mu_1,\mu_2) \ge d(\mu_0,\mu_2)
	\]
for any $\mu_0, \mu_1, \mu_2 \in M(\mR^n)$. Denote the probability vector associated with $\mu_0, \mu_1, \mu_2$ by $p_0, p_1, p_2$ respectively. The Gaussian components of $\mu_i$ is denoted by $\nu_i^j$. Let $\pi_{01}$ ($\pi_{12}$) be the solution to \eqref{eq:OMTdiscrete} with marginals $\mu_0, \mu_1$ ($\mu_1,\mu_2$). Define $\pi_{02}$ by
	\[
		\pi_{02}(i,k) = \sum_j \frac{\pi_{01}(i,j)\pi_{12}(j,k)}{p_1^j}.
	\]
Clearly, $\pi_{02}$ is a joint distribution between $p_0$ and $p_2$, namely, $\pi_{02}\in\Pi(p_0,p_2)$. It follows from direct calculation
	\begin{eqnarray*}
		\sum_i \pi_{02}(i,k) &=& \sum_{i,j} \frac{\pi_{01}(i,j)\pi_{12}(j,k)}{p_1^j}
		\\&=& \sum_{j} \frac{p_2^j\pi_{12}(j,k)}{p_1^j}
		\\&=& p_2^k.
	\end{eqnarray*}
Similarly, we have $\sum_{k} \pi_{02}(i,k)=p_0^i$.
Therefore,
	\begin{eqnarray*}
		d(\mu_0,\mu_2) &\le& \sqrt{\sum_{i,k}  \pi_{02}(i,k)W_2(\nu_0^i,\nu_2^k)^2}
		\\&=&
		\sqrt{\sum_{i,j,k}  \frac{\pi_{01}(i,j)\pi_{12}(j,k)}{p_1^j}W_2(\nu_0^i,\nu_2^k)^2}
		\\&\le&
		\sqrt{\sum_{i,j,k}  \frac{\pi_{01}(i,j)\pi_{12}(j,k)}{p_1^j}(W_2(\nu_0^i,\nu_1^j)+W_2(\nu_1^i,\nu_2^k))^2}
		\\&\le&
		\sqrt{\sum_{i,j,k}  \frac{\pi_{01}(i,j)\pi_{12}(j,k)}{p_1^j}W_2(\nu_0^i,\nu_1^j)^2}+
		\sqrt{\sum_{i,j,k}  \frac{\pi_{01}(i,j)\pi_{12}(j,k)}{p_1^j}W_2(\nu_1^j,\nu_2^k)^2}
		\\&=&
		\sqrt{\sum_{i,j}  \pi_{01}(i,j)W_2(\nu_0^i,\nu_1^j)^2}+
		\sqrt{\sum_{j,k} \pi_{12}(j,k)W_2(\nu_1^j,\nu_2^k)^2}
		\\&=& d(\mu_0,\mu_1) +d(\mu_1,\mu_2).
	\end{eqnarray*}
In the above, the second inequality is due to the fact $W_2$ is a metric, and the third inequality is an application of the Minkowski inequality. 	
\end{proof}

\subsection{Geodesic}
A geodesic on $M(\mR^n)$ connecting $\mu_0$ and $\mu_1$ is given by
	\begin{equation}\label{eq:geodesic}
	\mu_t = \sum_{i,j} \pi^*(i,j) \nu_t^{ij},
	\end{equation}
where $\nu_t^{ij}$ is the displacement interpolation (see \eqref{eq:disinterpG}) between $\nu_0^i$ and $\nu_1^j$.
\begin{thm}
	\begin{equation}\label{eq:st}
		d(\mu_s,\mu_t) = (t-s)d(\mu_0,\mu_1),\quad 0\le s < t\le 1.
	\end{equation}
\end{thm}
\begin{proof}
For any $0\le s\le t\le1$, we have
	\begin{eqnarray*}
	d(\mu_s,\mu_t) &\le& \sqrt{\sum_{i,j}\pi^*(i,j)W_2(\nu_s^{ij},\nu_t^{ij})^2}
	\\&=&(t-s)\sqrt{\sum_{i,j}\pi^*(i,j)W_2(\nu_0^{i},\nu_1^{j})^2} = (t-s) d(\mu_0,\mu_1)
	\end{eqnarray*}
where we have used the property \eqref{eq:W2geodesic} of $W_2$. It follows that
	\[
		d(\mu_0,\mu_s)+d(\mu_s,\mu_t)+d(\mu_t,\mu_1) \le sd(\mu_0,\mu_1)+(t-s)d(\mu_0,\mu_1)+(1-t)d(\mu_0,\mu_1)
		=d(\mu_0,\mu_1).
	\]
On the other hand, by Theorem \ref{thm:metric}, we have
	\[
		d(\mu_0,\mu_s)+d(\mu_s,\mu_t)+d(\mu_t,\mu_1)\ge d(\mu_0,\mu_1).
	\]
Combining these two, we obtain \eqref{eq:st}.
\end{proof}
We remark that $\mu_t$ is a Gaussian mixture model since it is a weighted average of the Gaussian distributions $\nu_t^{ij}$.
Even though the solution to (\ref{eq:OMTdiscrete}) is not unique in some instances, it is unique for generic $\mu_0,\mu_1\in M(\mR^n)$. Therefore, in most real applications, we need not worry about the uniqueness.

\subsection{Relation between $d$ and $W_2$}
We first note that we have
	\[
		d(\mu_0,\mu_1) \ge W_2(\mu_0,\mu_1)
	\]
for any $\mu_0,\mu_1\in M(\mR^n)$.
Equality holds when both $\mu_0$ and $\mu_1$ have only one Gaussian component. In general, $d>W_2$. This is due to the fact that the restriction to the submanifold $M(\mR^n)$ induces sub-optimality in the transport plan.
Let $\gamma(t), 0\le t\le 1$ be any piecewise smooth curve  on $M(\mR^n)$ connecting $\mu_0$ and $\mu_1$. Define the Wasserstein length of $\gamma$ by
	\[
		L_W(\gamma) =\sup_{0=t_0<t_1<\cdots<t_s=1} \sum_{k} W_2(\gamma_{t_k},\gamma_{t_{k+1}}),
	\]
and natural length by
	\[
		L(\gamma) =\sup_{0=t_0<t_1<\cdots<t_s=1} \sum_{k} d(\gamma_{t_k},\gamma_{t_{k+1}}).
	\]
Then $L_W(\gamma) \le L(\gamma)$.

Using the metric property of $d$ we get
	\[
		d(\mu_0,\mu_1) \le \inf_\gamma L(\gamma),
	\]
where the minimization is taken over all the piecewise smooth curve  on $M(\mR^n)$ connecting $\mu_0$ and $\mu_1$. In view of \eqref{eq:st}, we conclude
	\[
		d(\mu_0,\mu_1) = \inf_\gamma L(\gamma) \ge \inf_\gamma L_W(\gamma).
	\]
Therefore, it is unclear whether$d$ is the restriction of $W_2$ to $M(\mR^n)$.

In general, $d$ is a very good approximation of $W_2$ if the variances of the Gaussian components are small compared with the differences between the means. This may lead to an efficient algorithm to approximate Wasserstein distance between two distributions with such properties.
%In particular, $d(\mu_0,\mu_1)$ converges to $W_2(\mu_0,\mu_1)$ if we let the variances of the Gaussian components go to zero.
%Given two distributions $\mu_0,\mu_1\in M(\mR^n)$, we would like to compute their difference or distance.
If we want to compute the Wasserstein distance $W_2(\mu_0,\mu_1)$ between two distributions $\mu_0,\mu_1\in M(\mR^n)$, a standard procedure is discretizing the densities first, and then solving a discrete OMT problem. Depending upon the resolution of the discretization, the second step may become very costly. In contrast, to compute our new distance $d(\mu_0,\mu_1)$, we need only to solve (\ref{eq:OMTdiscrete}). When the number of Gaussian components of $\mu_0,\mu_1$ are small, this is extremely efficient.

\section{Barycenter of Gaussian mixtures}
The barycenter \cite{AguCar11} of $L$ distributions $\mu_0,\mu_1,\ldots,\mu_L$ is defined to be the minimizer of
	\begin{equation}
		J(\mu)= \frac{1}{L}\sum_{k=1}^L W_2(\mu,\mu_k)^2.
	\end{equation}
This resembles the average $\frac{1}{L}(x_1+x_2+\cdots+x_L)$ of $L$ points in the Euclidean space, which minimizes
	\[
		J(x) = \frac{1}{L}\sum_{k=1} \|x-x_k\|^2.
	\]
The above definition can be generalized to the cost
	\begin{equation}\label{eq:barycenter}
		\min_{\mu\in P_2(\mR^n)} \sum_{k=1}^L \lambda_kW_2(\mu,\mu_k)^2.
	\end{equation}
where $\lambda = [\lambda_1,\lambda_2,\ldots,\lambda_L]$ is a probability vector. The existence and uniqueness of \eqref{eq:barycenter} has been extensively studied in \cite{AguCar11} where it is shown that under some mild assumptions, the solution exists and is unique.

In the special case when all $\mu_k$ are Gaussian distributions, the barycenter remains Gaussian. In particular, denoting the mean and covariance of $\mu_k$ as $m_k, \Sigma_k$, then the barycenter has mean
	\begin{equation}
	m = \sum_{k=1}^L \lambda_k m_k
	\end{equation}
and covariance $\Sigma$ solving
	\begin{equation}\label{eq:barycov}
		\Sigma = \sum_{k=1}^L \lambda_k (\Sigma^{1/2} \Sigma_k\Sigma^{1/2})^{1/2}.
	\end{equation}
A fast algorithm to get the solution of \eqref{eq:barycov} is through the fixed point iteration \cite{AlvDelCueMat16}
	\[
		(\Sigma)_{\rm next} =
		 \Sigma^{-1/2}\left(\sum_{k=1}^L \lambda_k (\Sigma^{1/2} \Sigma_k\Sigma^{1/2})^{1/2}\right)^2\Sigma^{-1/2}.
	\]		
In practice, the iteration
	\[
		(\Sigma)_{\rm next} = \sum_{k=1}^L \lambda_k (\Sigma^{1/2} \Sigma_k\Sigma^{1/2})^{1/2}
	\]
appears to also work. However, no convergence proof for the latter is known at present \cite{AguCar11,AlvDelCueMat16}.

For general distributions, the barycenter problem \eqref{eq:barycenter} is difficult to solve. It can be reformulated as a multi-marginal optimal transport problem and is therefore convex. Recently several algorithms have been proposed to solve \eqref{eq:barycenter} through entropic regularization \cite{BenCarCut15}. However, due to the curse of dimensionality, solving such a problem for dimension greater than $3$ is still unrealistic. This is the case even for Gaussian mixture models. What's more, the Gaussian mixture structure is  often lost when solving problem \eqref{eq:barycenter}.

To overcome this issue for Gaussian mixtures, herein, we propose to solve a modified barycenter problem
	\begin{equation}\label{eq:baryGMM}
		\min_{\mu\in M(\mR^n)} \sum_{k=1}^L \lambda_kd(\mu,\mu_k)^2.
	\end{equation}
The optimization variable is restricted to be Gaussian mixture distribution and the Wasserstein distance $W_2$ is replaced by
its relaxed version \eqref{eq:metric}. Let $\mu_k$ be a Gaussian mixture distribution with $N_k$ components, namely, $\mu_k = p_k^1 \nu_k^1 + p_k^2 \nu_k^2 + \cdots+p_k^{N_k} \nu_k^{N_k}$. If we view $\mu$ as a discrete measure on $G(\mR^n)$, then clearly, it can only have support at the points (Gaussian distributions) of the form
	\begin{equation}\label{eq:barycomp}
		\argmin_{\nu} \sum_{k=1}^L \lambda_kW_2(\nu,\nu_k^{i_k})^2
	\end{equation}
with $\nu_k^{i_k}$ being any component of $\mu_k$. As we discussed before, the optimal $\nu$ is Gaussian. Denote the set of all such minimizers as $\{\nu^1,\nu^2,\ldots,\nu^N\}$, then $\mu$ is equal to
	\[
		\mu = p^1 \nu^1 + p^2 \nu^2+\cdots+p^N\nu^N,
	\]
for some probability vector $p=(p^1,p^2,\ldots,p^N)^T$.
The number of element $N$ is bounded above by $N_1N_2\cdots N_L$.
Finally, utilizing the definition of $d(\cdot,\cdot)$ we obtain an equivalent formulation of \eqref{eq:baryGMM}, which reads as
	\begin{subequations}\label{eq:linearbary}
	\begin{eqnarray}
		&&\min_{\pi_1\ge 0,\cdots,\pi_L\ge 0} \sum_{k=1}^L\sum_{i=1}^N\sum_{j_k=1}^{N_k} \lambda_k c_k(i,j_k) \pi_k(i,j_k)
		\\
		&& \sum_{i=1}^N \pi_k(i,j_k) = p_k^{j_k},\quad \forall 1\le k\le L, 1\le j_k \le N_k
		\\
		&& \sum_{j_1=1}^{N_1} \pi_1(i,j_1)= \sum_{j_2=1}^{N_2} \pi_2(i,j_2)=\cdots=
		 \sum_{j_L=1}^{N_L} \pi_L(i,j_L),\quad \forall 1\le i\le N.
	\end{eqnarray}
	\end{subequations}
The cost
	\begin{equation}\label{eq:unitcost}
		c_k(i,j) = W_2(\nu^i,\nu_k^j)^2
	\end{equation}
is the optimal transport cost from $\nu^i$ to $\nu_k^j$. After solving the above linear programming problem \eqref{eq:linearbary}, we get the barycenter $\mu=p^1 \nu^1 + p^2 \nu^2+\cdots+p^N\nu^N$ with
	\[
		p^i = \sum_{j=1}^{N_1} \pi_1(i,j)
	\]
for each $1\le i\le N$. We remark that our formulation is independent of the dimension of the underlying space $\mR^n$. The dimension affects only the computation of the cost function \eqref{eq:unitcost} where a closed-form \eqref{eq:W2gaussian} is available. The complexity of \eqref{eq:linearbary} relies on the numbers of components of the Gaussian mixtures distributions $\{\mu_k\}$. Therefore, our formulation is extremely efficiently for high dimensional Gaussian mixtures with small number of components.

The difficulty of formulation \eqref{eq:linearbary} lies on the number $N$ of components of the barycenter $\mu$, which is usually of order $N_1N_2\cdots N_L$. To overcome this issue, we can consider the barycenter problem for Gaussian mixture with specified components. More specifically, given 	$N$ Gaussian components $\nu^1, \nu^2, \ldots, \nu^N$, we would like to find a minimizer of the optimization problem \eqref{eq:barycenter} subject to the structure constraint that
	\[
		\mu = p^1 \nu^1 + p^2 \nu^2+\cdots+p^N\nu^N
	\]
for some probability vector $p=(p^1,p^2,\ldots,p^N)^T$. Note that $\nu^k$ here doesn't have to be of the form \eqref{eq:barycomp}. It can be any Gaussian distribution. Moreover, the number $N$ can be chosen to be small. It turns out  this problem can be solved in exactly the same way. Clearly, a linear programming reformulation \eqref{eq:linearbary} is straightforward.

\section{Numerical Examples}
Several examples are provided to illustrate our framework in computing distance, geodesic and barycenter.
\subsection{$d$ vs $W_2$}
To demonstrate the difference between $d$ and $W_2$, we investigate a simple example here. We choose $\mu_0$ to be a zero-mean Gaussian distribution with unit variance. The terminal distribution $\mu_1$ is set to be the average of two unit variance Gaussian distributions, one with mean $\Delta$ and the other one with mean $-\Delta$. Clearly, $d(\mu_0,\mu_1)=\Delta$. Figure \ref{fig:dvsW2} depicts $d(\mu_0,\mu_1)$ and $W_2(\mu_0,\mu_1)$ for different $\Delta$ values. As can be seen, these two distances are not equivalent and $d$ is always bounded below by $W_2$.
\begin{figure}[h]
\centering
\includegraphics[width=0.4\textwidth]{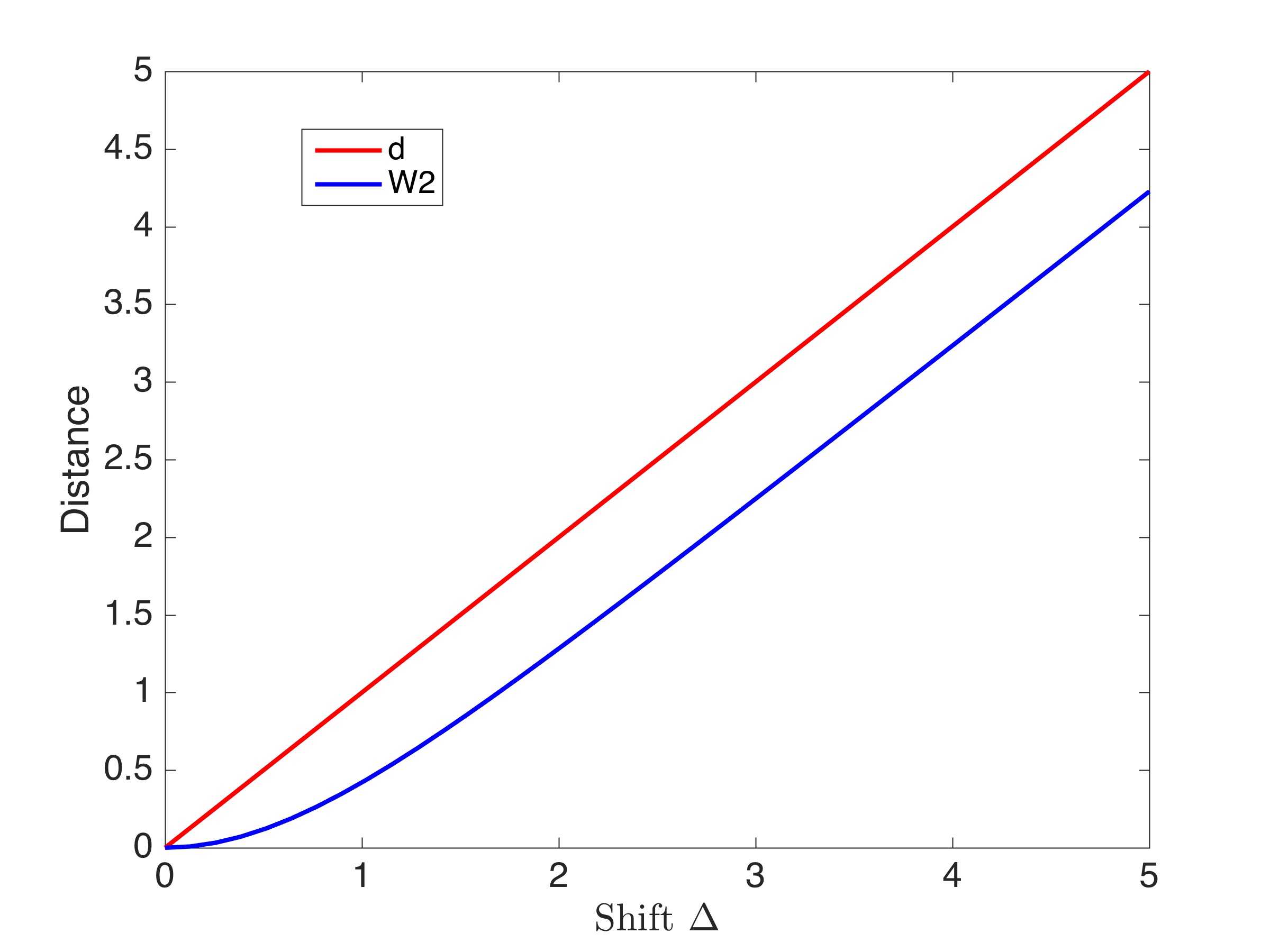}
\caption{$d$ vs $W_2$}
\label{fig:dvsW2}
\end{figure}

\subsection{Geodesic}
We compare the displacement interpolation in standard OMT theory, and our proposed geodesic interpolation (see \eqref{eq:geodesic}) on $M(\mR^n)$.
Consider the two Gaussian mixture models in Figure~\ref{fig:eg1marginals}. Both of them have two components; one in red, one in blue and the mixture model is in black. For both marginals, the masses are equally distributed among the components. The means and covariances are $m_0^1=0.5, m_0^2=0.1,\Sigma_0^1=0.01,\Sigma_0^2=0.05$ for $\mu_0$, and $m_1^1=0, m_1^2=-0.35,\Sigma_1^1=0.02,\Sigma_1^2=0.02$ for $\mu_1$. Figure~\ref{fig:eg1interp} depicts the interpolation results based on standard OMT and our method. As we can see from the figures, the intermediate densities based on standard OMT interpolation lose the Gaussian mixture structure. This is not the case for our method. The two Gaussian components of the interpolation based on proposed method described are shown in Figure~\ref{fig:eg1gaussian}.

We have similar observations for an example on 2-dimensional spaces; see Figures~\ref{fig:eg1twodmarginals}-\ref{fig:eg1twodinterp}. The two marginal distributions are Gaussian mixtures shown in Figure~\ref{fig:eg1twodmarginals}. Figure~\ref{fig:eg1twodomt} and Figure~\ref{fig:eg1twodinterp} are two interpolations, based on OMT and our method, respectively. We can see that the Gaussian mixture structure is undermined in Figure~\ref{fig:eg1twodomt}.
\begin{figure}[h]
\centering
\subfloat[$\mu_0$]{\includegraphics[width=0.40\textwidth]{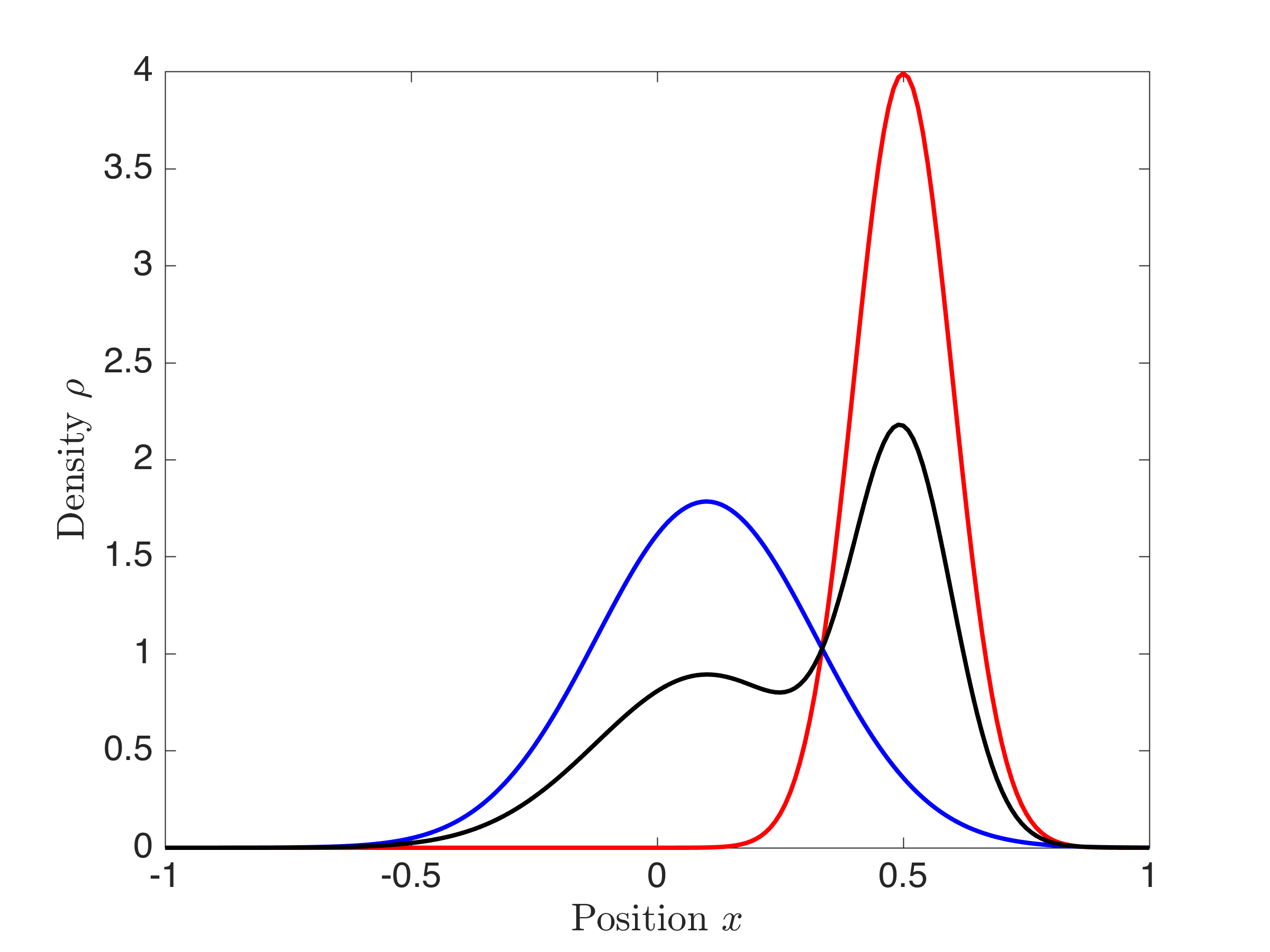}}
\subfloat[$\mu_1$]{\includegraphics[width=0.40\textwidth]{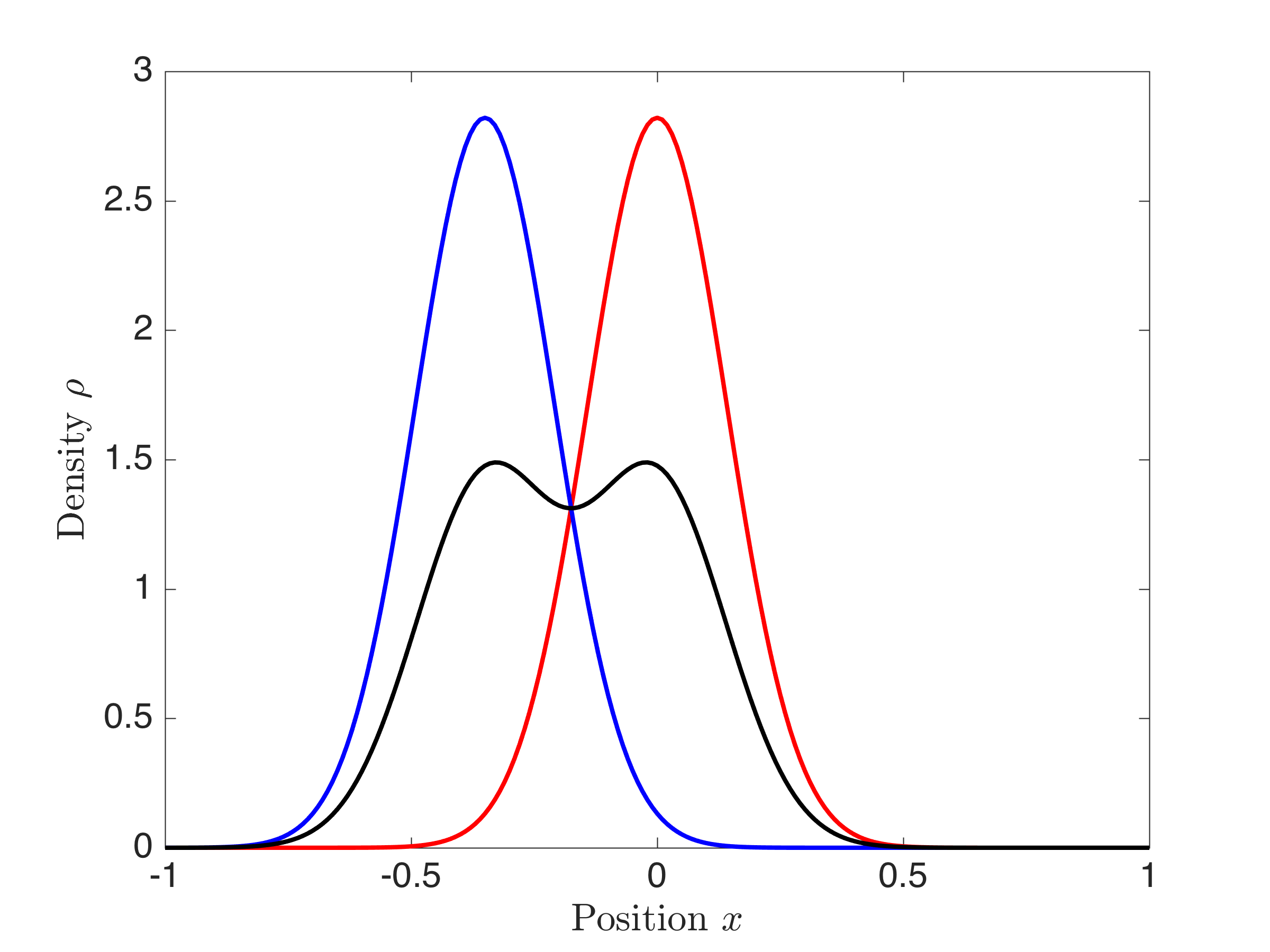}}
 \caption{Marginal distributions}
 \label{fig:eg1marginals}
\end{figure}
\begin{figure}[h]
\centering
\subfloat[OMT]{\includegraphics[width=0.40\textwidth]{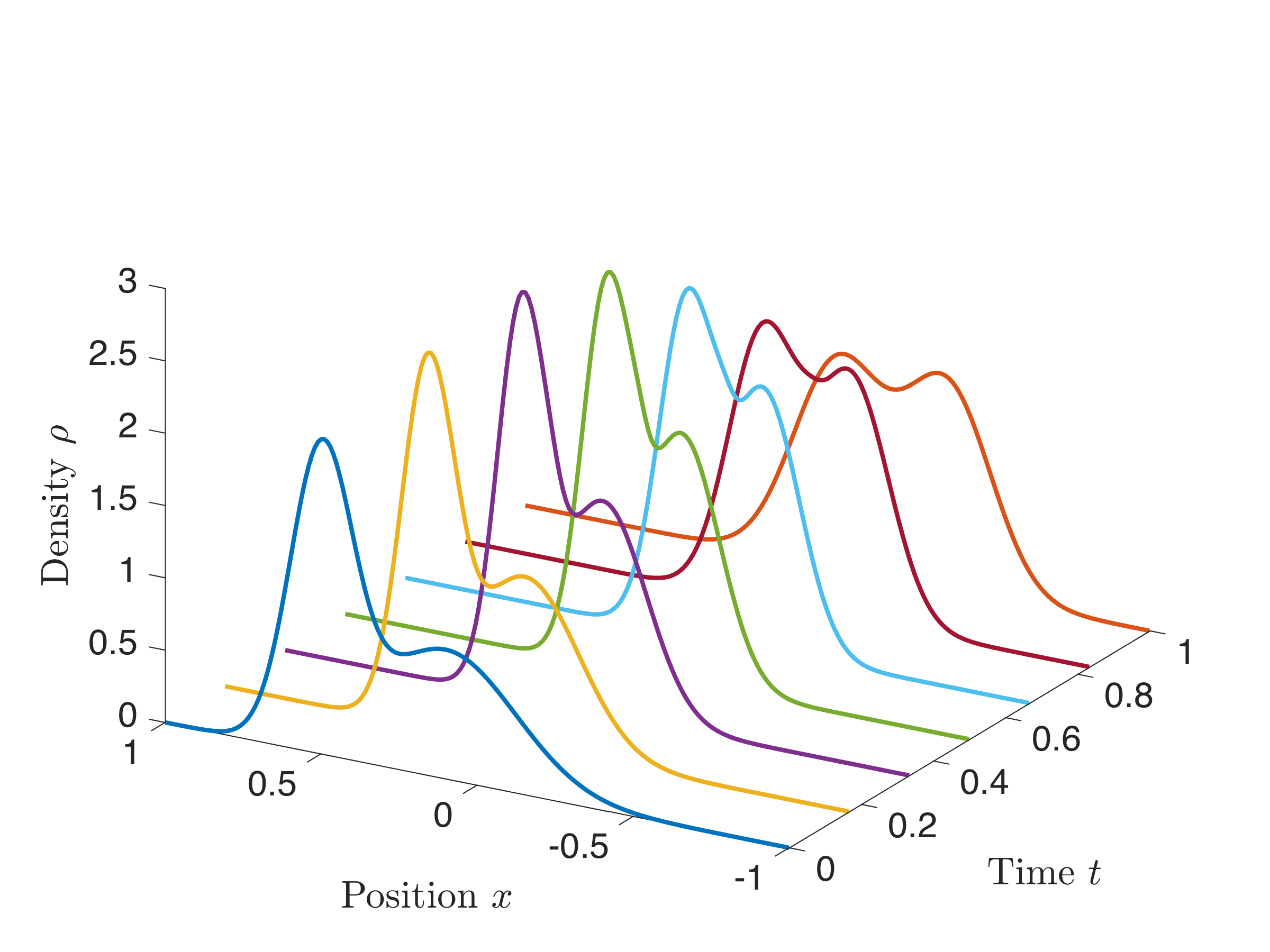}}
\subfloat[our framework]{\includegraphics[width=0.40\textwidth]{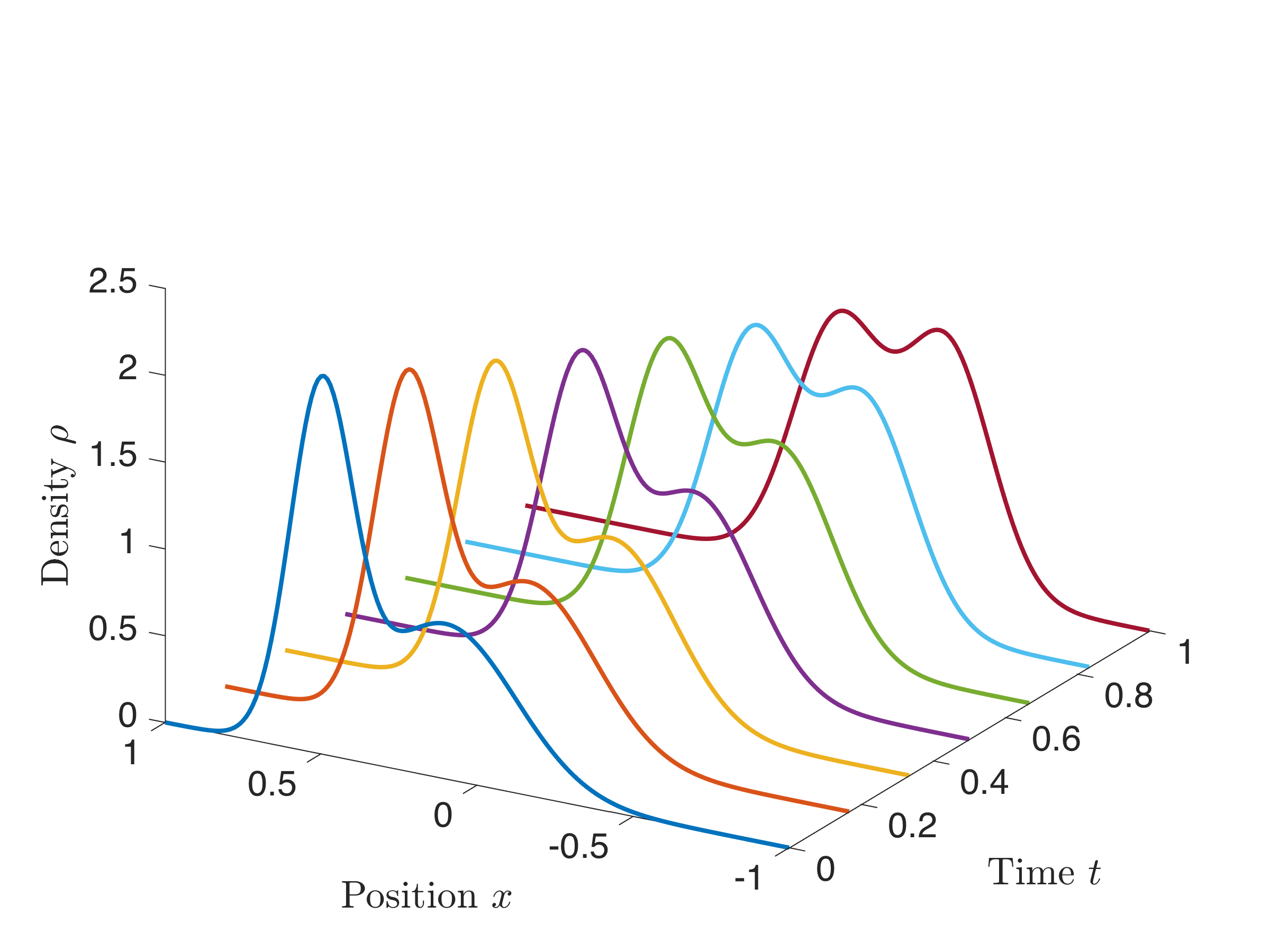}}
 \caption{Interpolations}
 \label{fig:eg1interp}
\end{figure}
\begin{figure}[h]
\centering
\includegraphics[width=0.4\textwidth]{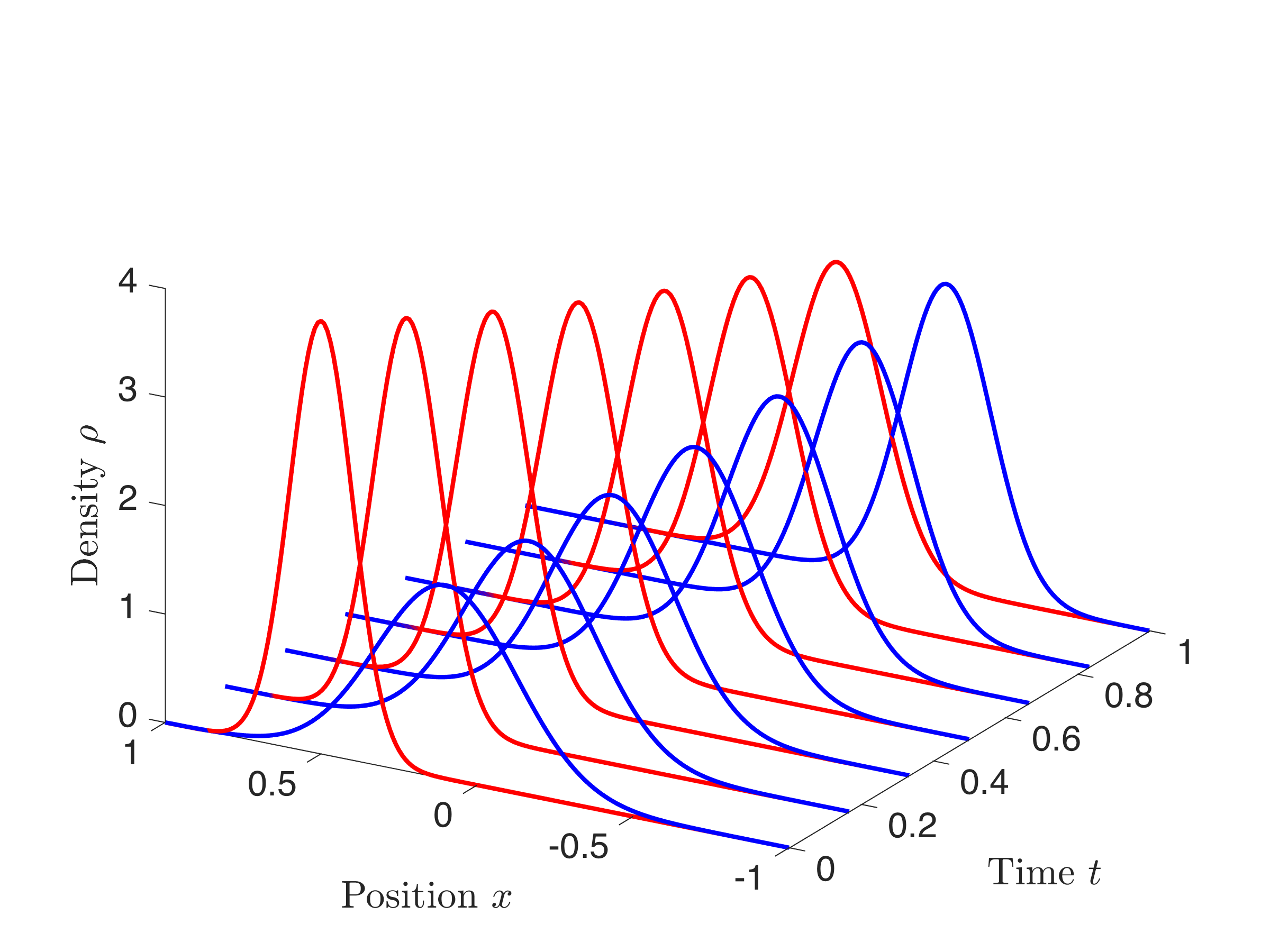}
\caption{Two Gaussian components of the interpolation}
\label{fig:eg1gaussian}
\end{figure}
\begin{figure}[h]
\centering
\subfloat[$\mu_0$]{\includegraphics[width=0.40\textwidth]{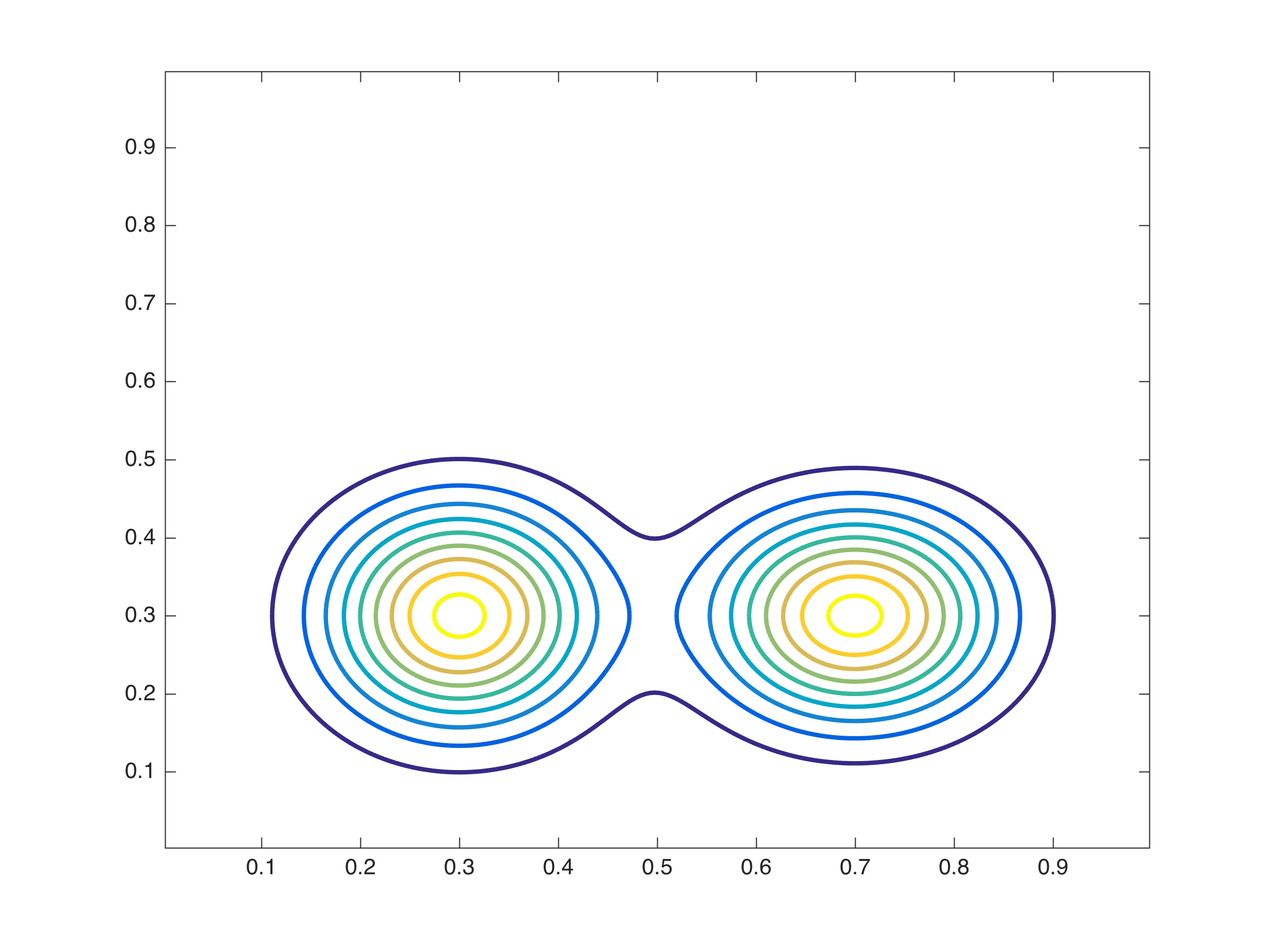}}
\subfloat[$\mu_1$]{\includegraphics[width=0.40\textwidth]{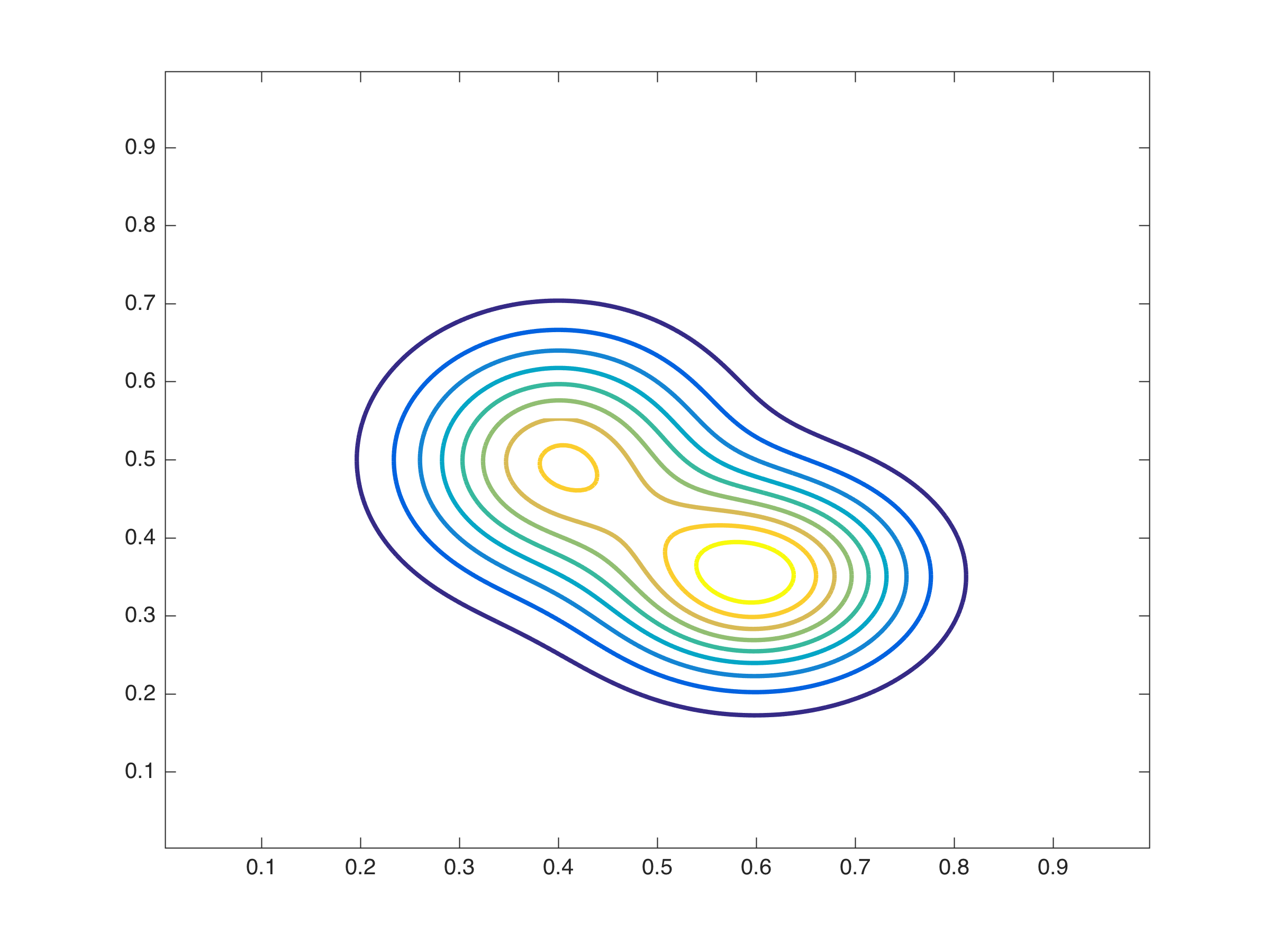}}
 \caption{Marginal distributions}
 \label{fig:eg1twodmarginals}
\end{figure}
\begin{figure}[h]
\centering
\subfloat{\includegraphics[width=0.20\textwidth]{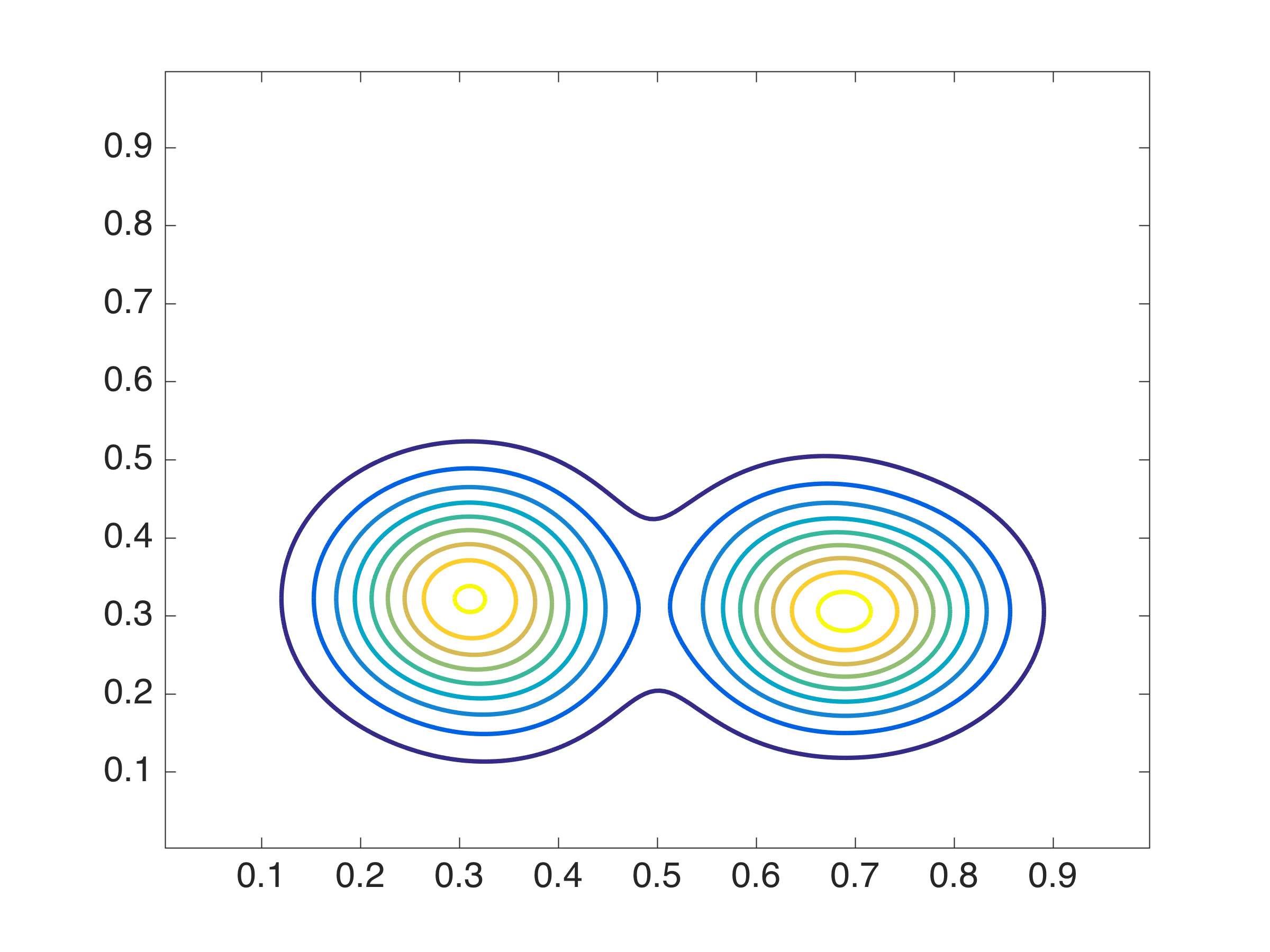}}
\subfloat{\includegraphics[width=0.20\textwidth]{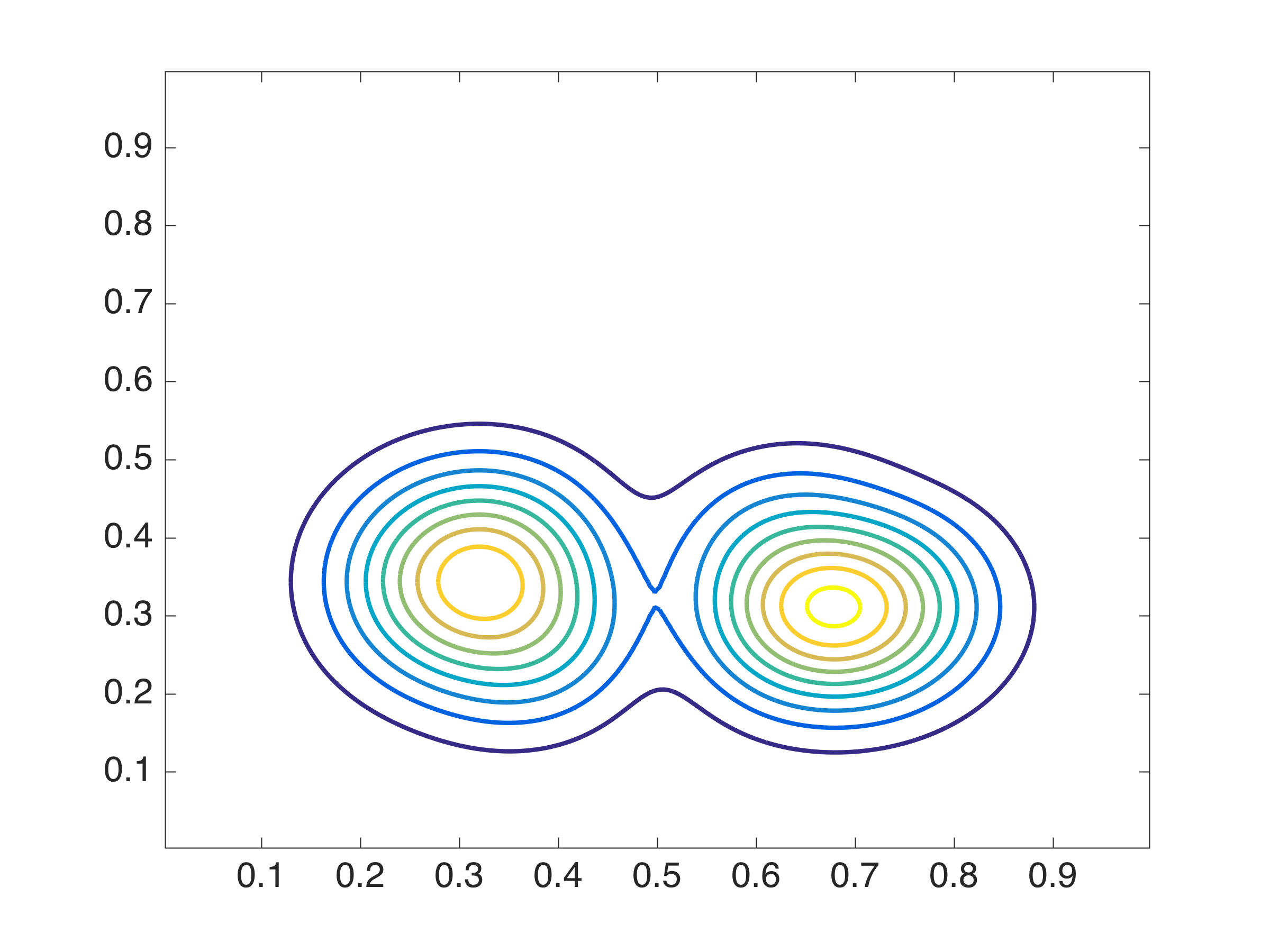}}
\subfloat{\includegraphics[width=0.20\textwidth]{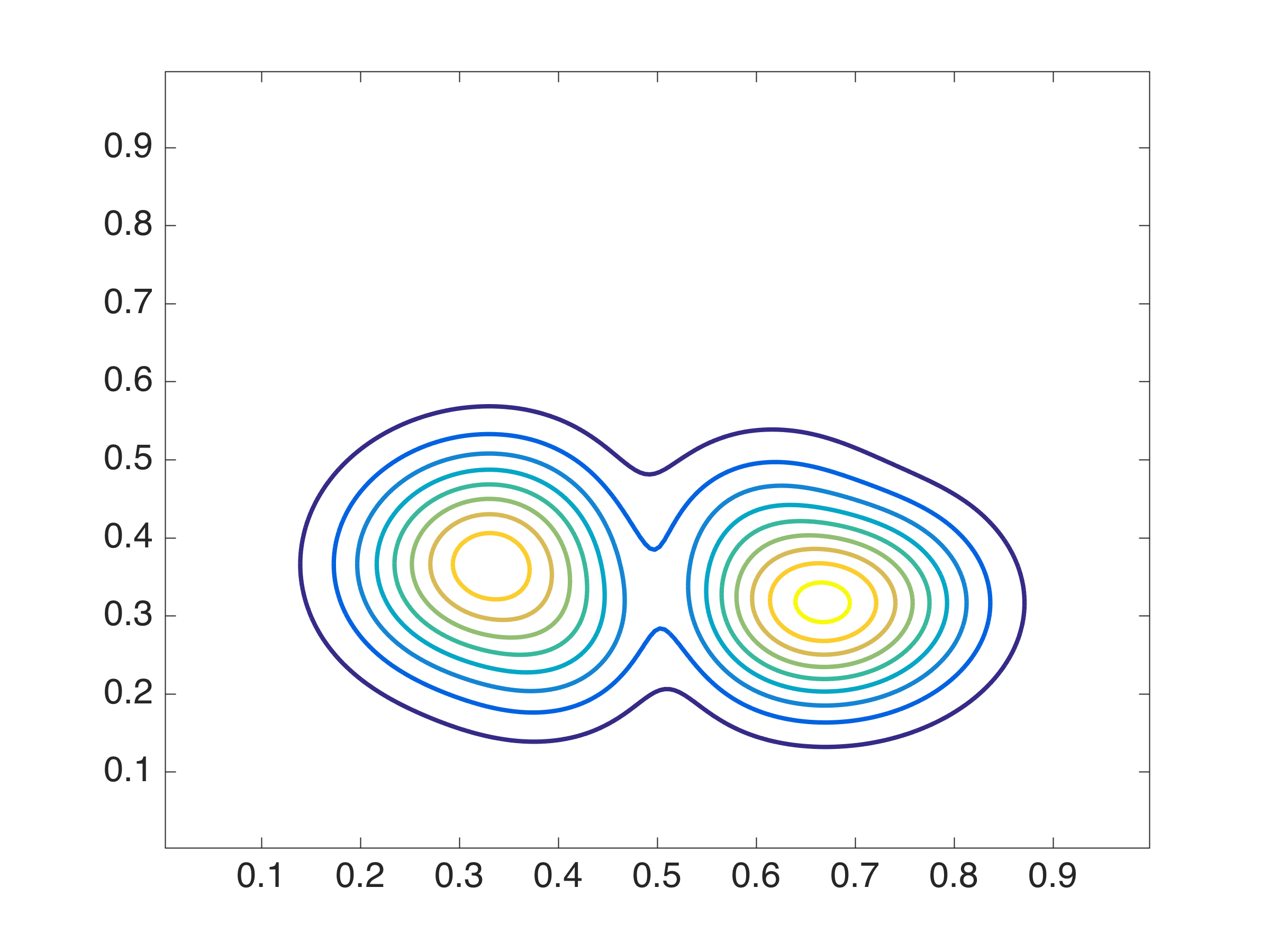}}
\subfloat{\includegraphics[width=0.20\textwidth]{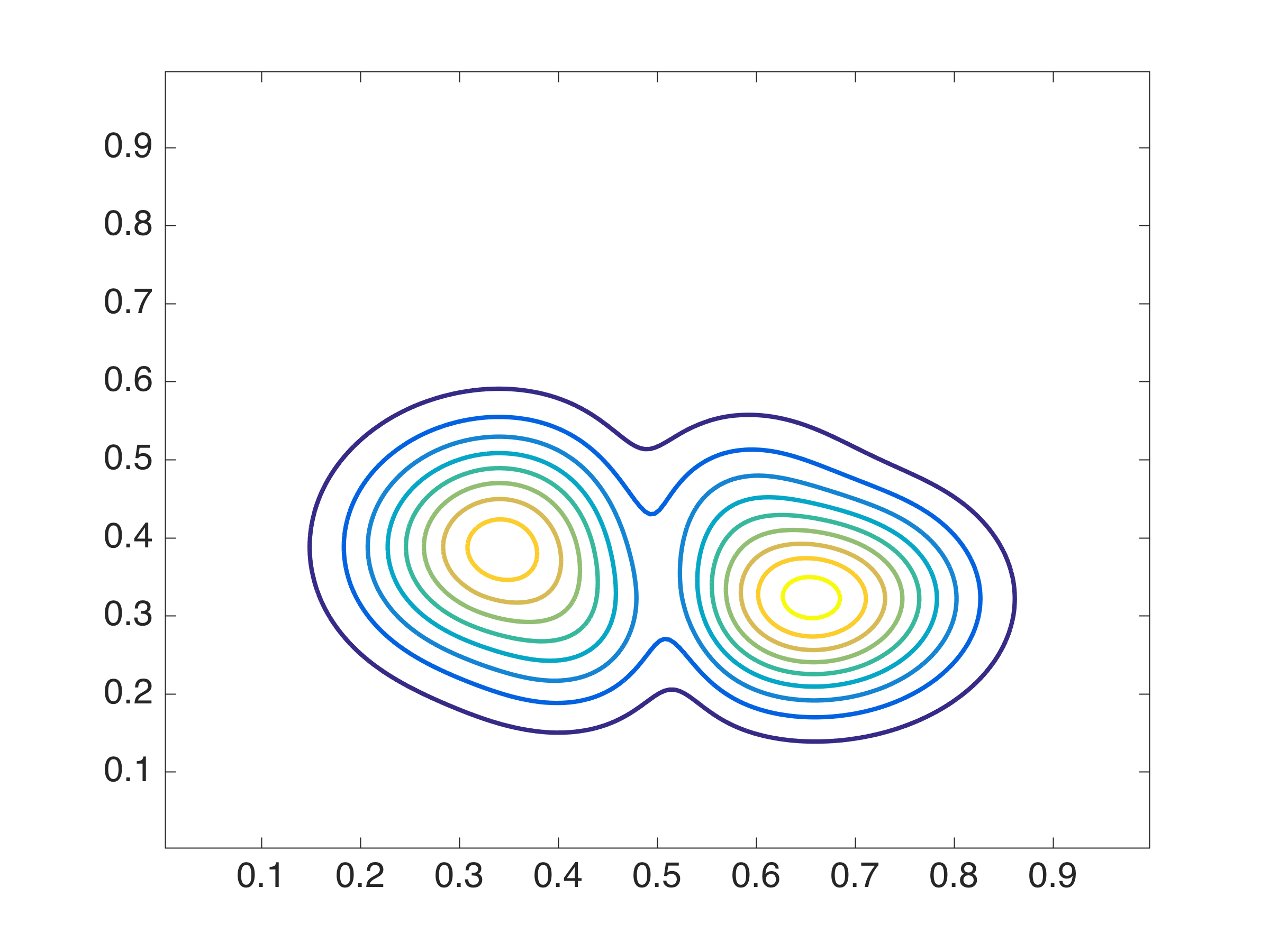}}
\\
\subfloat{\includegraphics[width=0.20\textwidth]{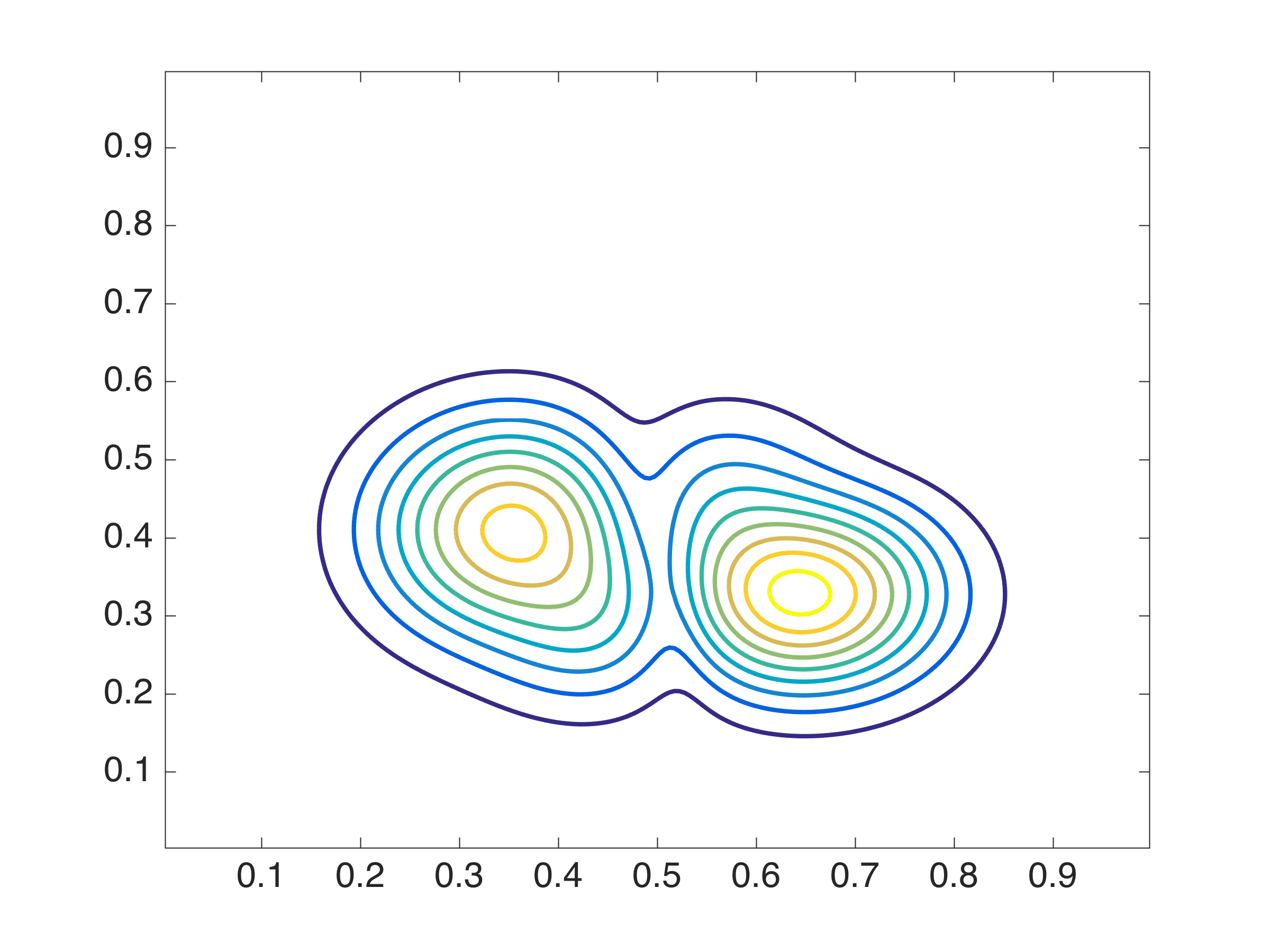}}
\subfloat{\includegraphics[width=0.20\textwidth]{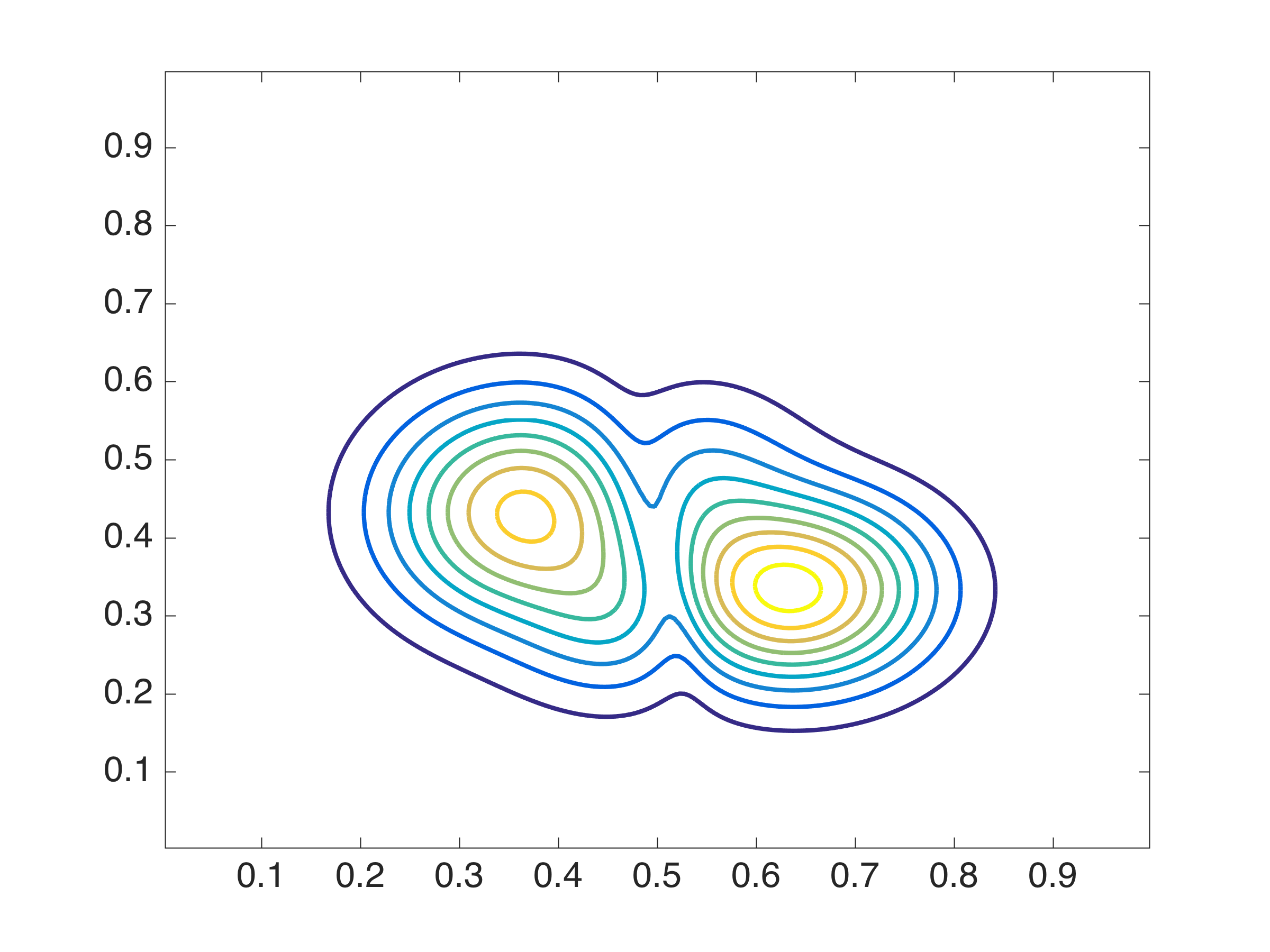}}
\subfloat{\includegraphics[width=0.20\textwidth]{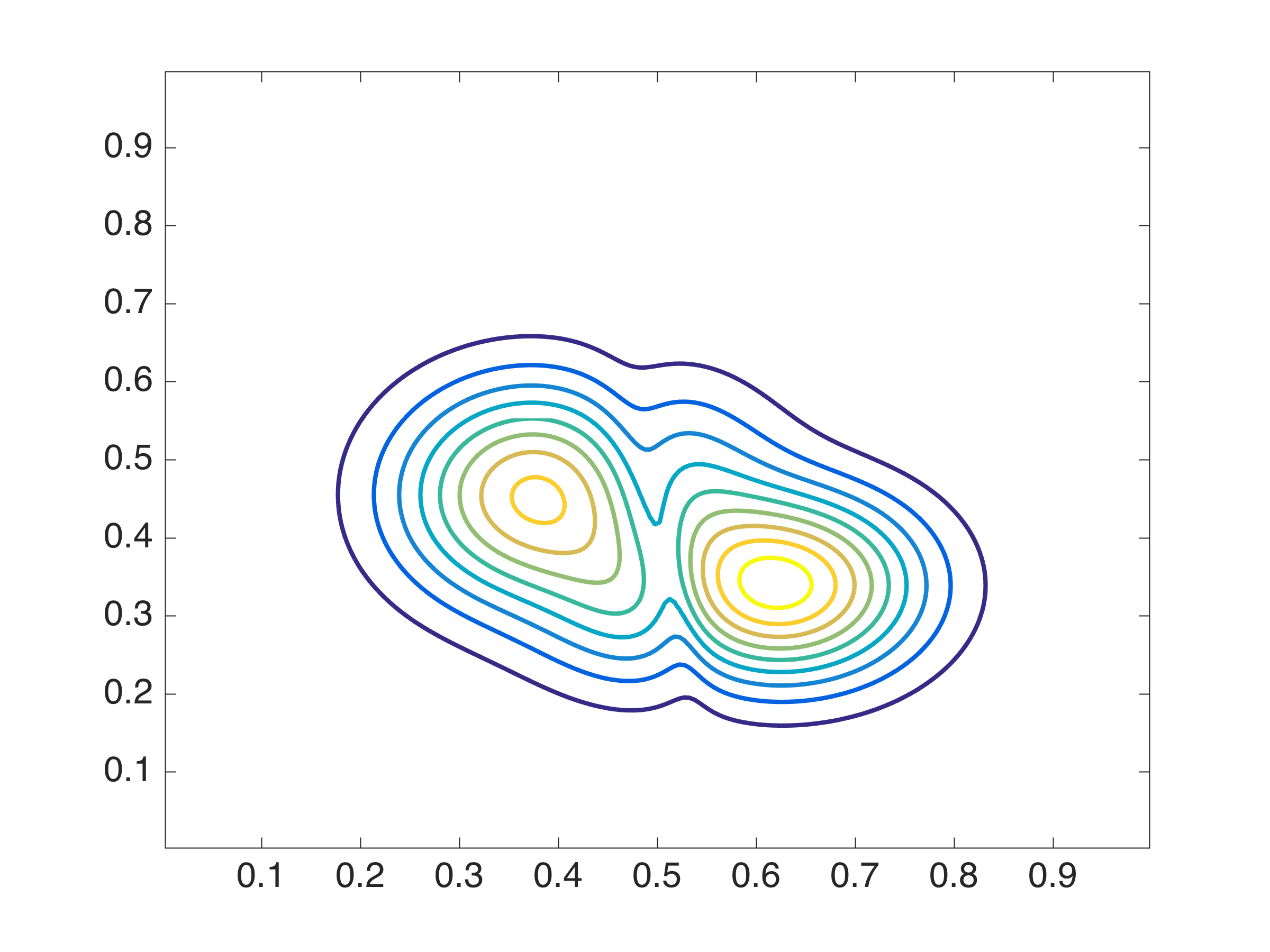}}
\subfloat{\includegraphics[width=0.20\textwidth]{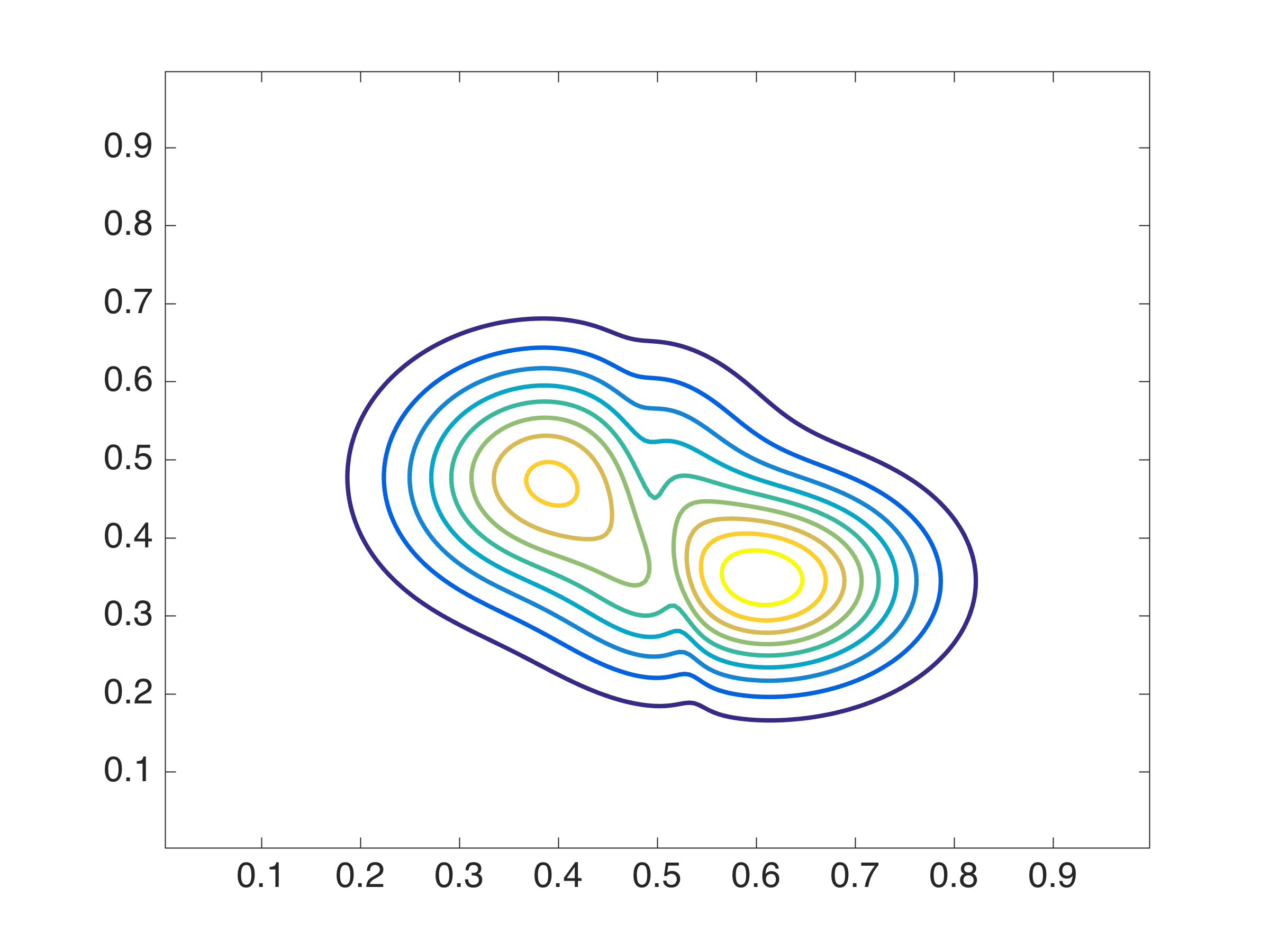}}
 \caption{OMT Interpolation}
 \label{fig:eg1twodomt}
\end{figure}
\begin{figure}[h]
\centering
\subfloat{\includegraphics[width=0.20\textwidth]{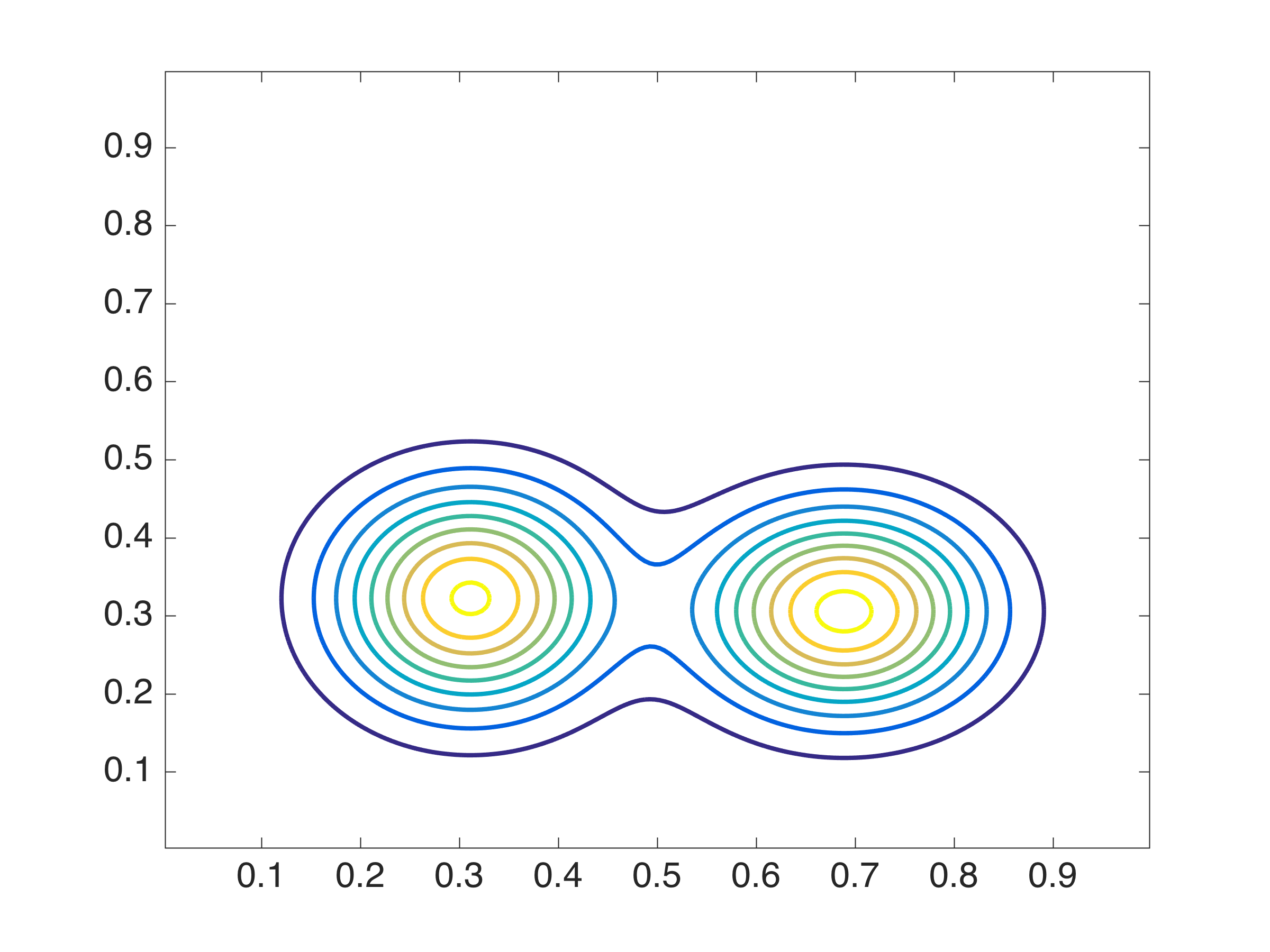}}
\subfloat{\includegraphics[width=0.20\textwidth]{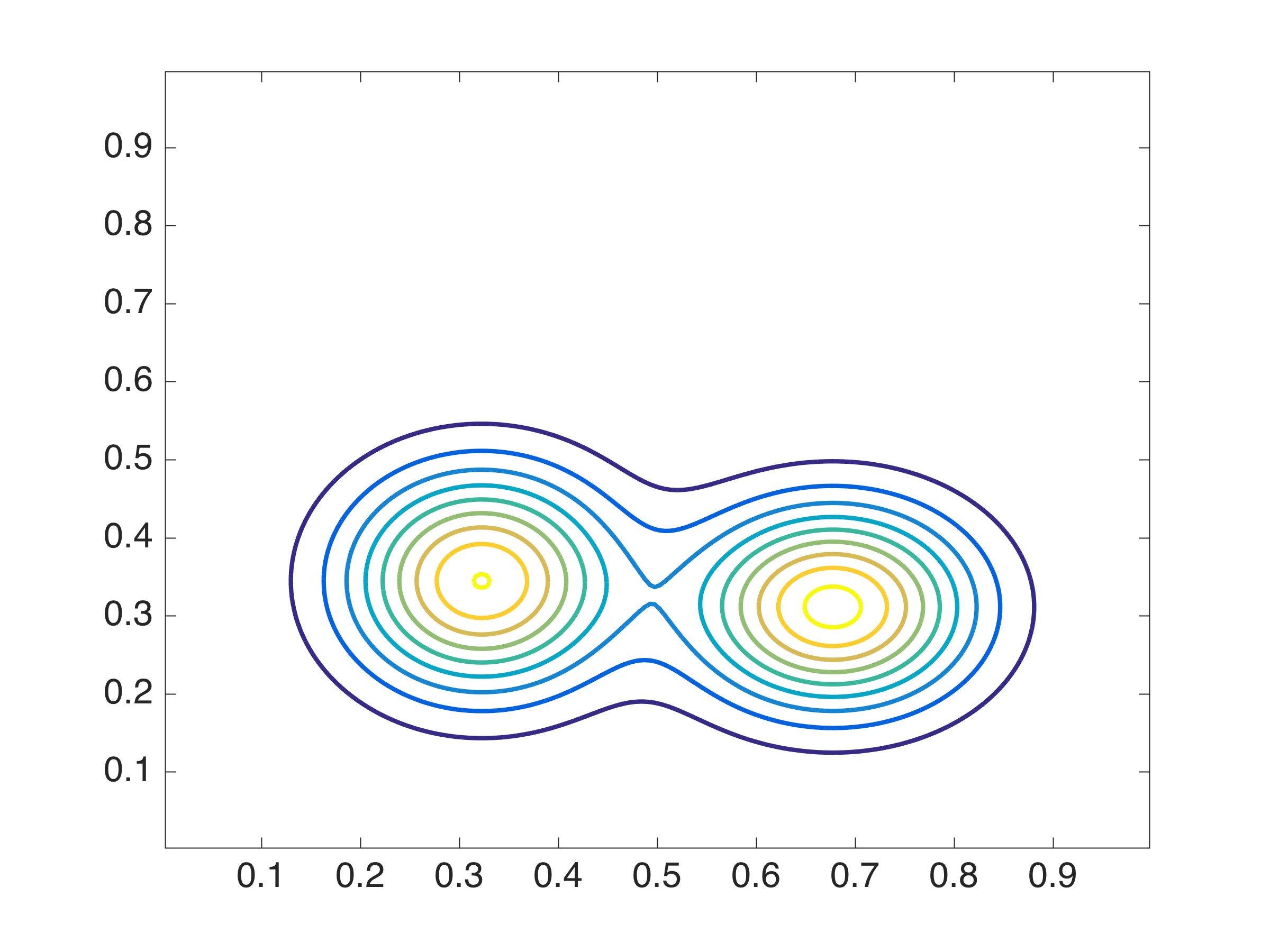}}
\subfloat{\includegraphics[width=0.20\textwidth]{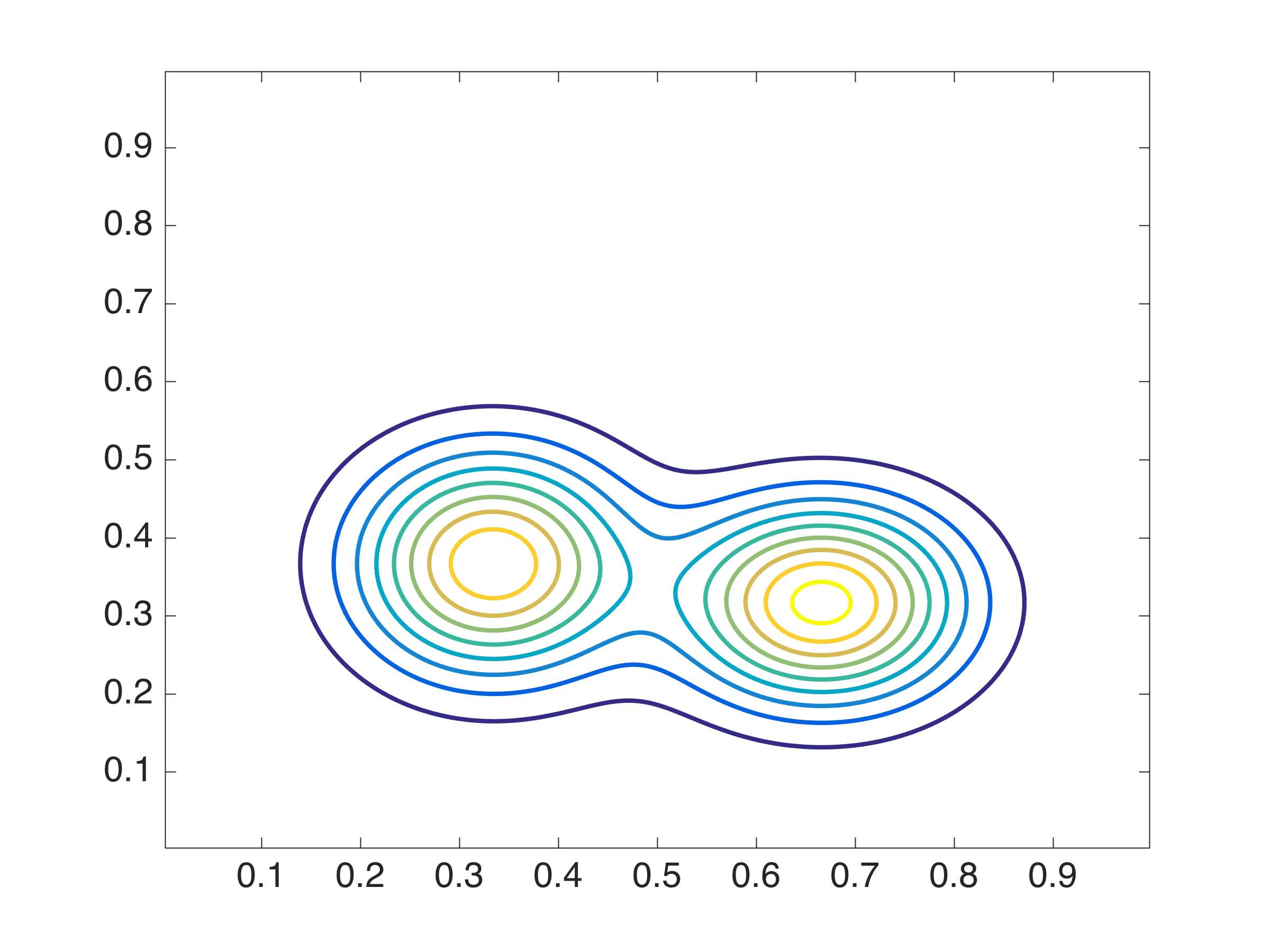}}
\subfloat{\includegraphics[width=0.20\textwidth]{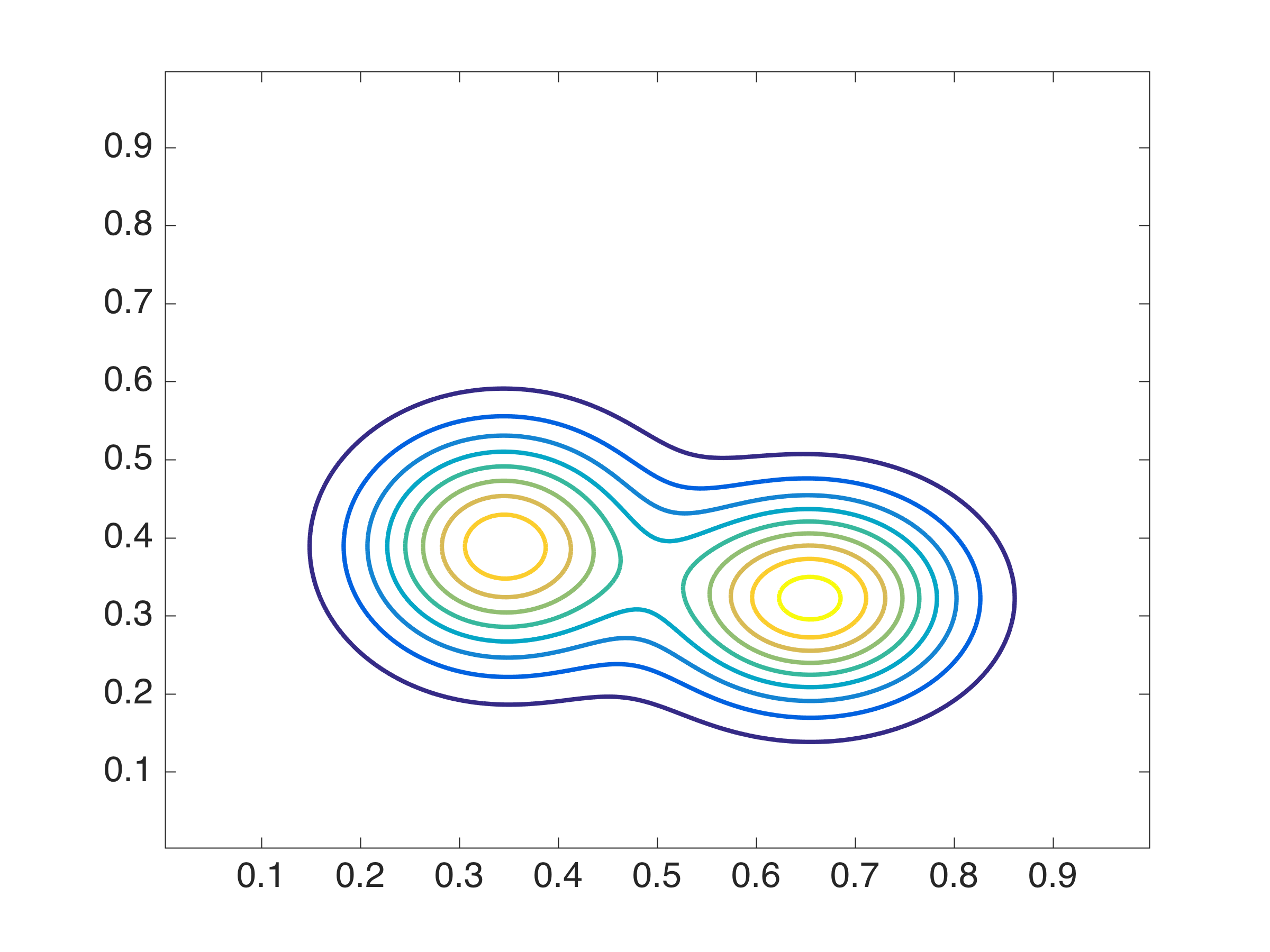}}
\\
\subfloat{\includegraphics[width=0.20\textwidth]{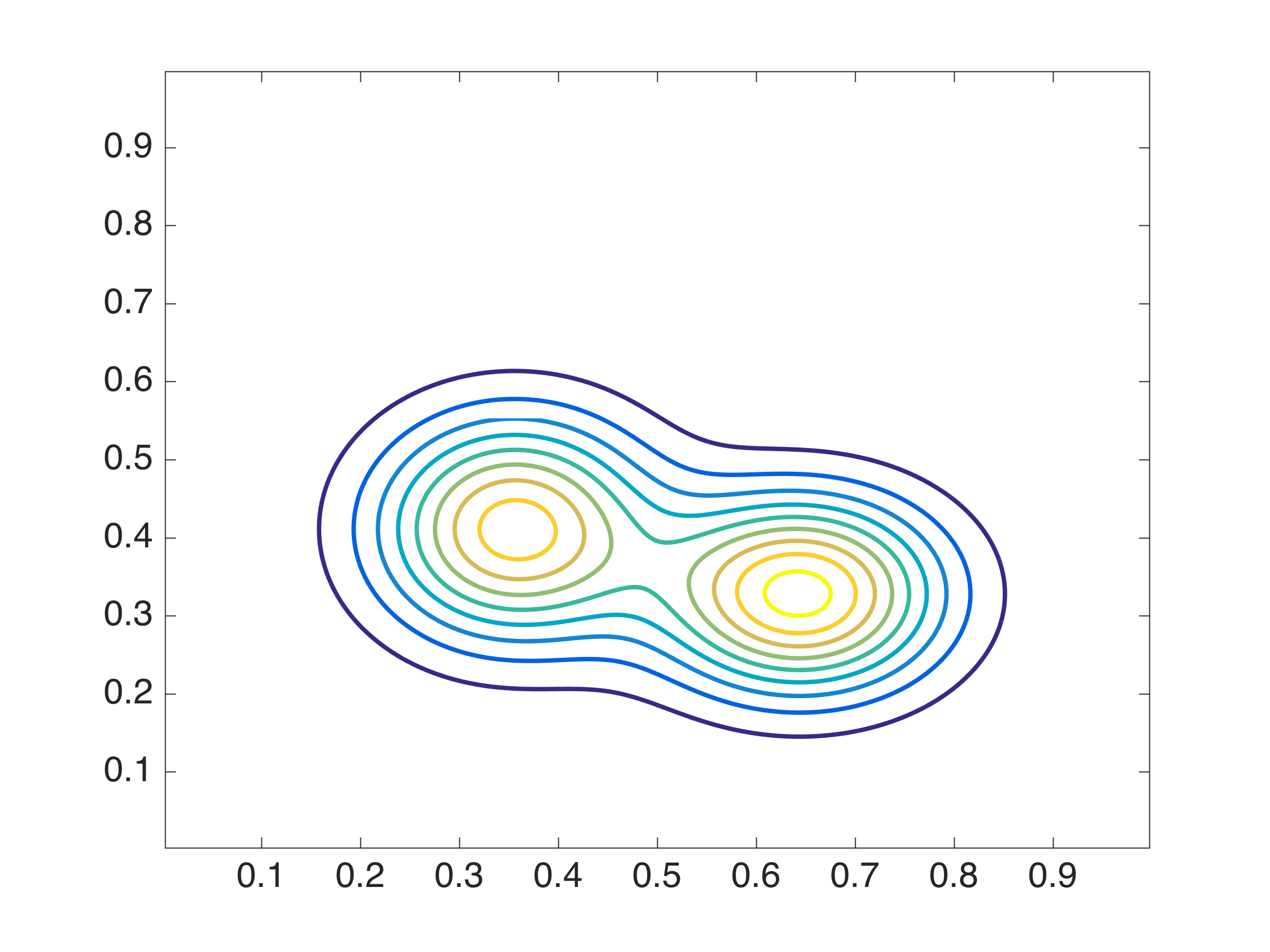}}
\subfloat{\includegraphics[width=0.20\textwidth]{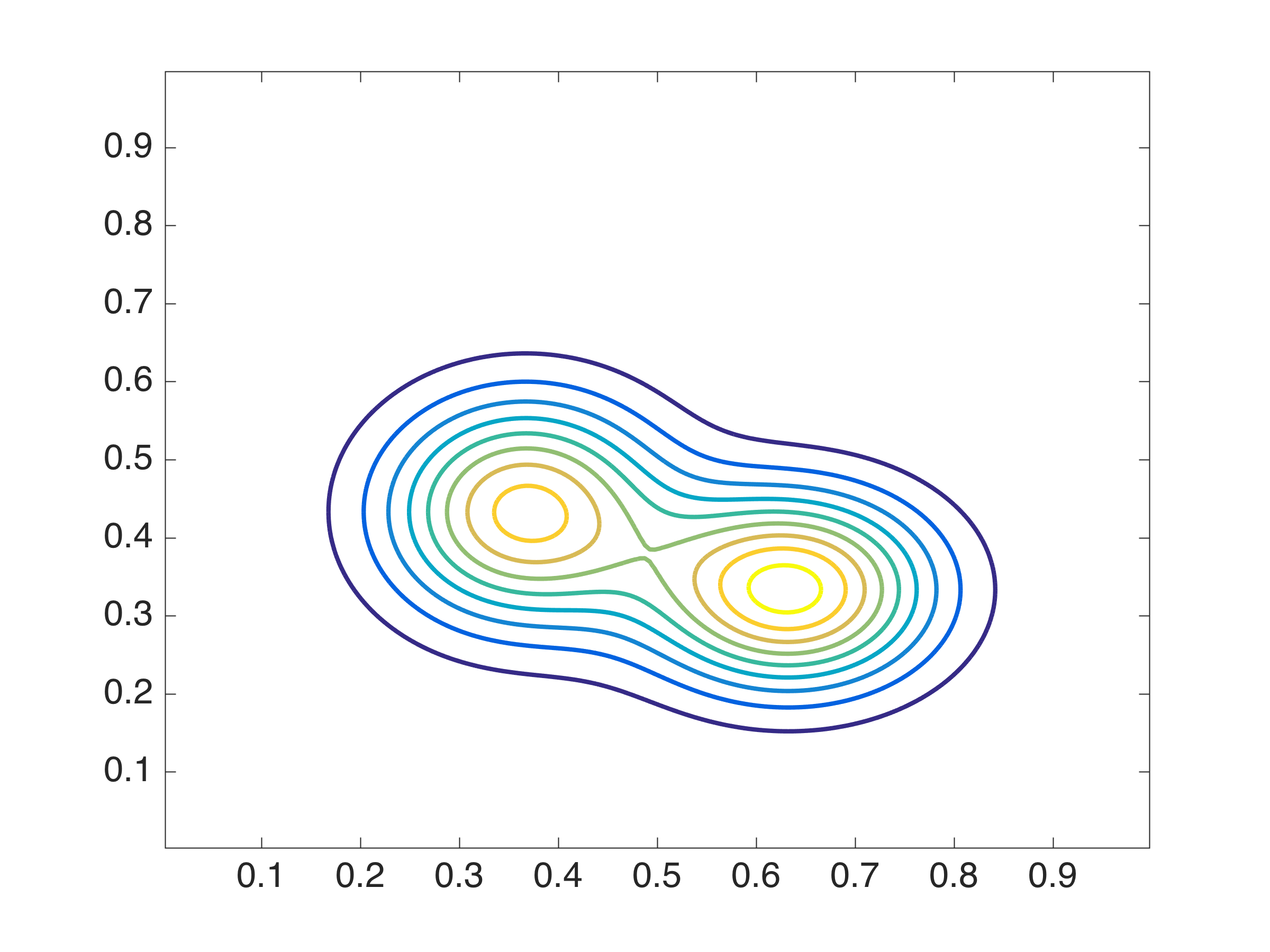}}
\subfloat{\includegraphics[width=0.20\textwidth]{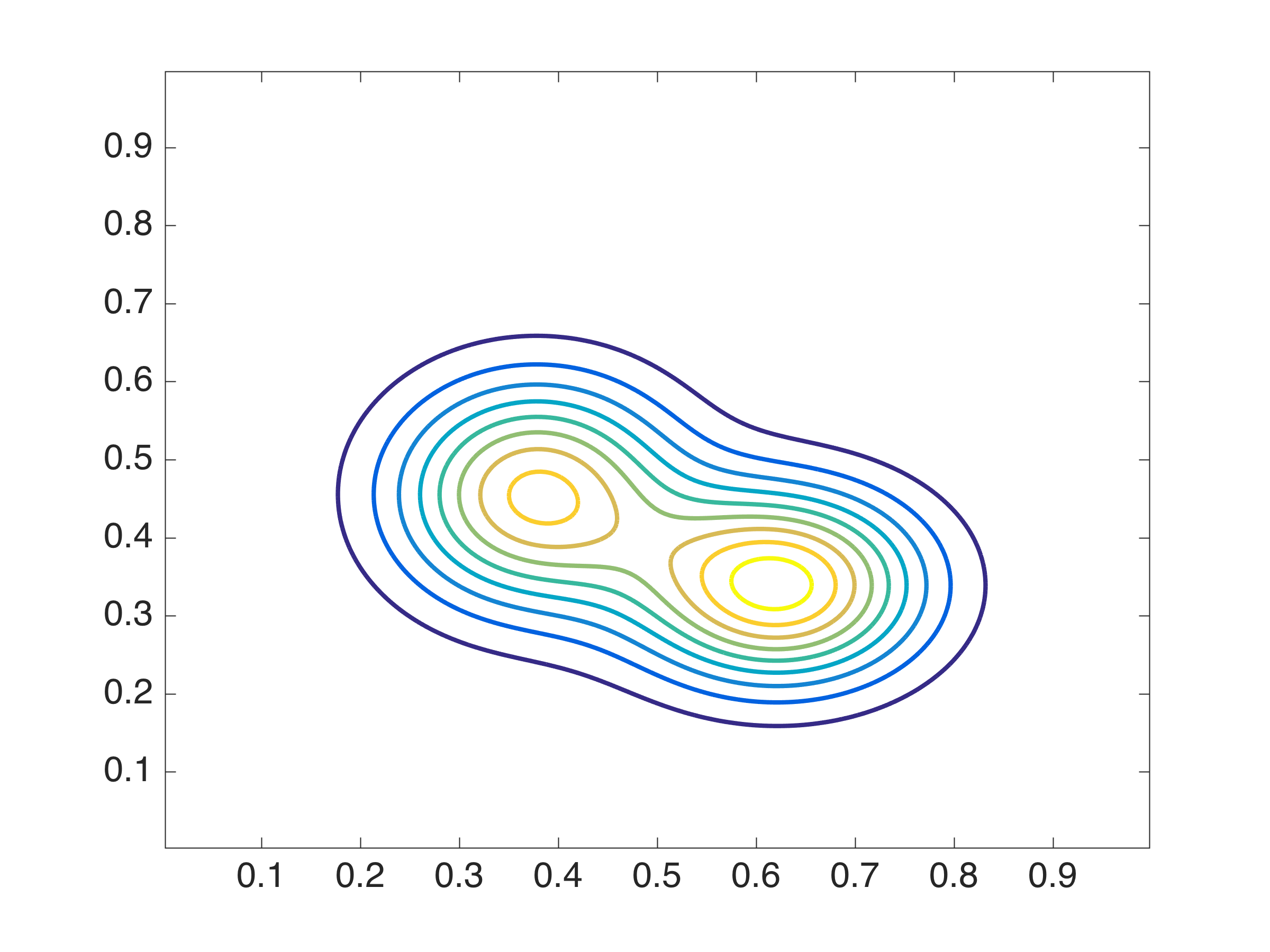}}
\subfloat{\includegraphics[width=0.20\textwidth]{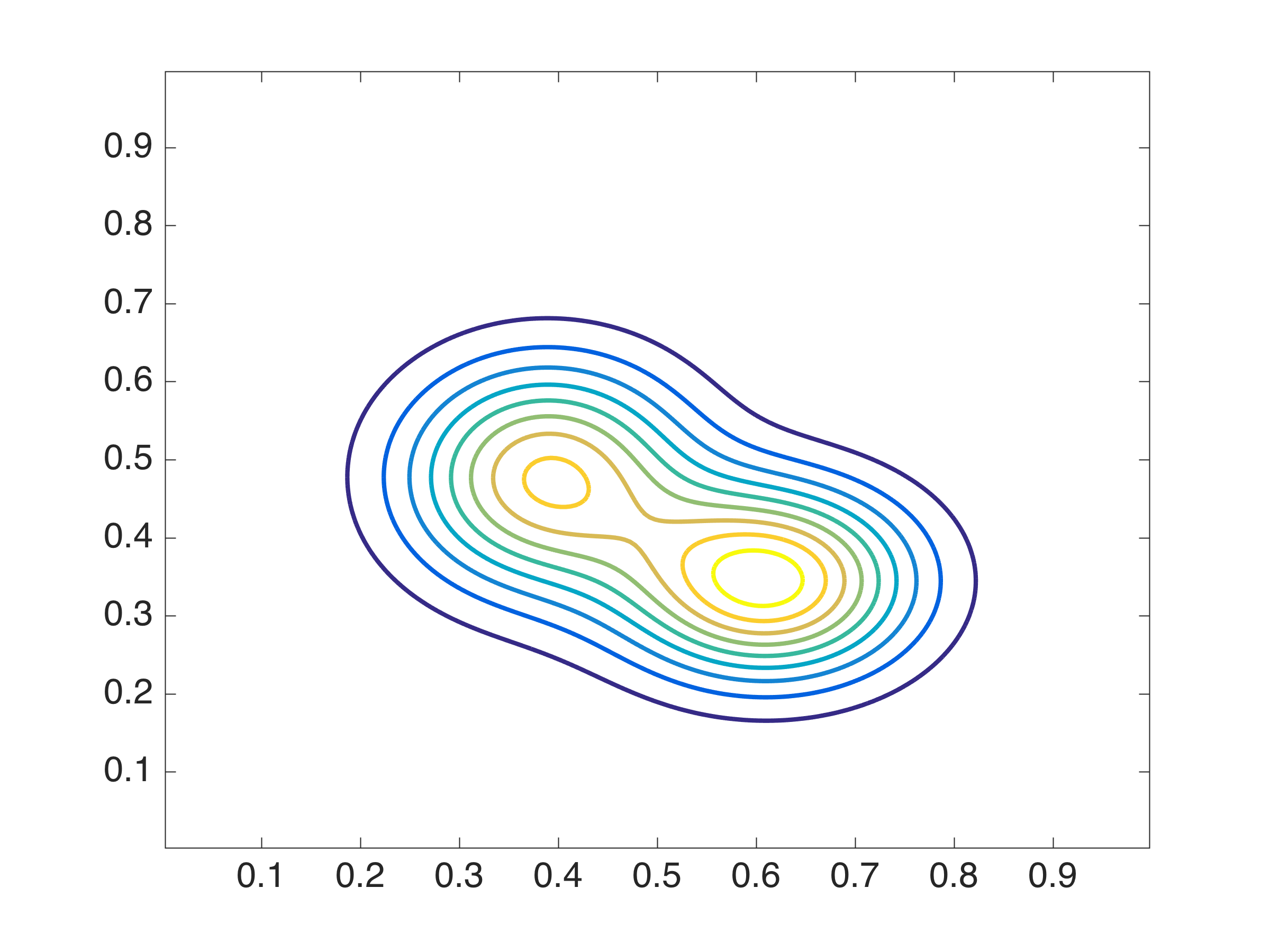}}
 \caption{Our Interpolation}
 \label{fig:eg1twodinterp}
\end{figure}
\subsection{Barycenter}
As shown in Figure \ref{fig:eg3marginals}, three Gaussian mixture distributions are given. The masses are equally distributed among the components. The statistics for $\mu_1, \mu_2, \mu_3$ are $(m_1^1=0.5, m_1^2=0.1,\Sigma_1^1=0.01,\Sigma_1^2=0.05)$, $(m_2^1=0, m_2^2=-0.35,\Sigma_2^1=0.02,\Sigma_2^2=0.02)$ and $(m_3^1=0.4, m_3^2=-0.45,\Sigma_3^1=0.025,\Sigma_3^2=0.021)$ respectively. We apply both our method and the traditional OMT theory to compute the barycenter. Two sets of weights are considered and the results are displayed in Figure \ref{fig:eg3barycenter} and \ref{fig:eg3barycenter1}. Apparently, our method gives better average. The Gaussian mixtures structure is damaged if traditional OMT is used.
\begin{figure}[h]
\centering
\subfloat[$\mu_1$]{\includegraphics[width=0.30\textwidth]{eg1rho0.png}}
\subfloat[$\mu_2$]{\includegraphics[width=0.30\textwidth]{eg1rho1.png}}
\subfloat[$\mu_3$]{\includegraphics[width=0.30\textwidth]{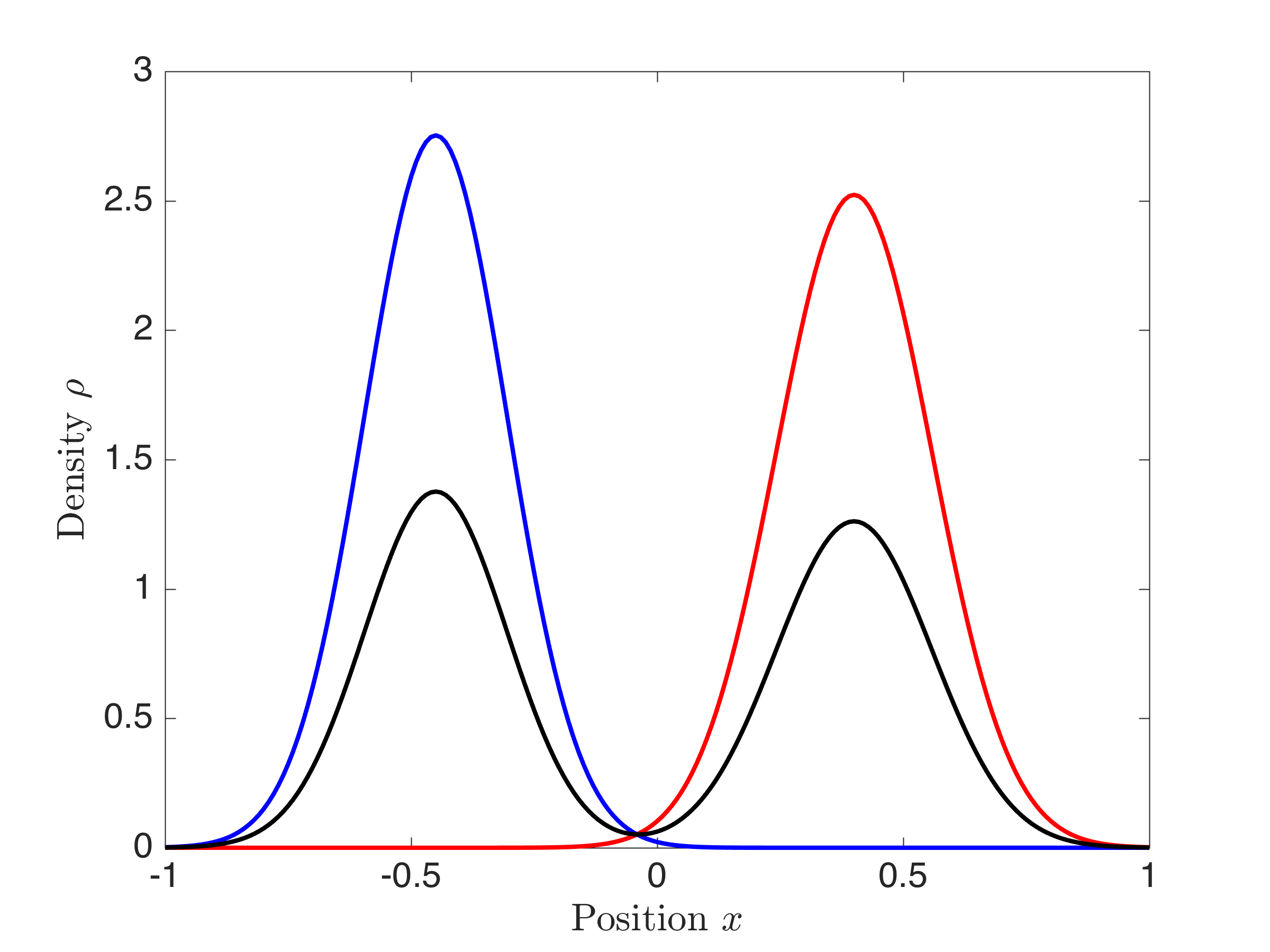}}
 \caption{Marginal distributions}
 \label{fig:eg3marginals}
\end{figure}
\begin{figure}[h]
\centering
\subfloat[our method]{\includegraphics[width=0.40\textwidth]{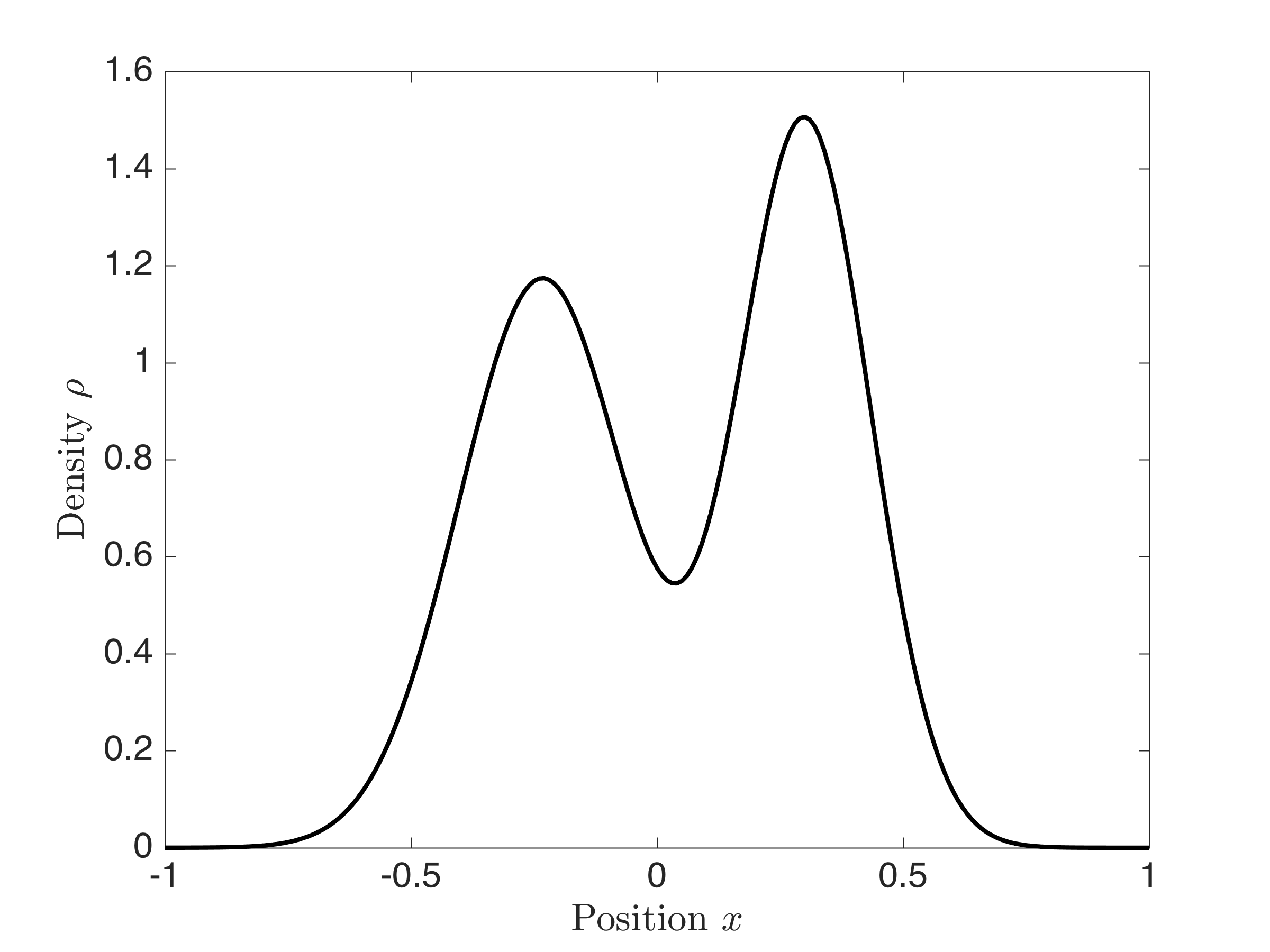}}
\subfloat[optimal transport]{\includegraphics[width=0.40\textwidth]{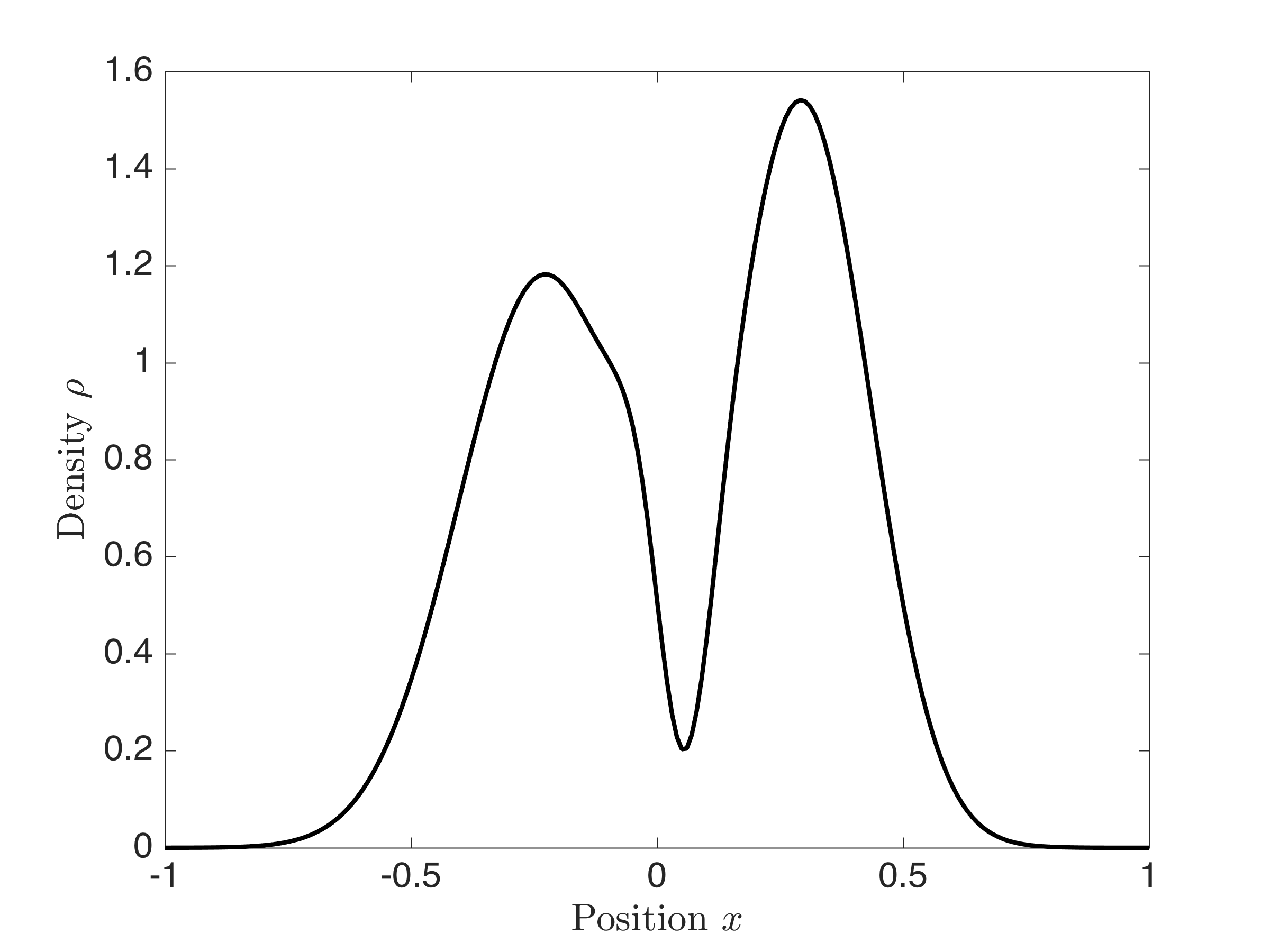}}
 \caption{Barycenters with $\lambda = (1/3, 1/3, 1/3)$}
 \label{fig:eg3barycenter}
\end{figure}
\begin{figure}[h]
\centering
\subfloat[our method]{\includegraphics[width=0.40\textwidth]{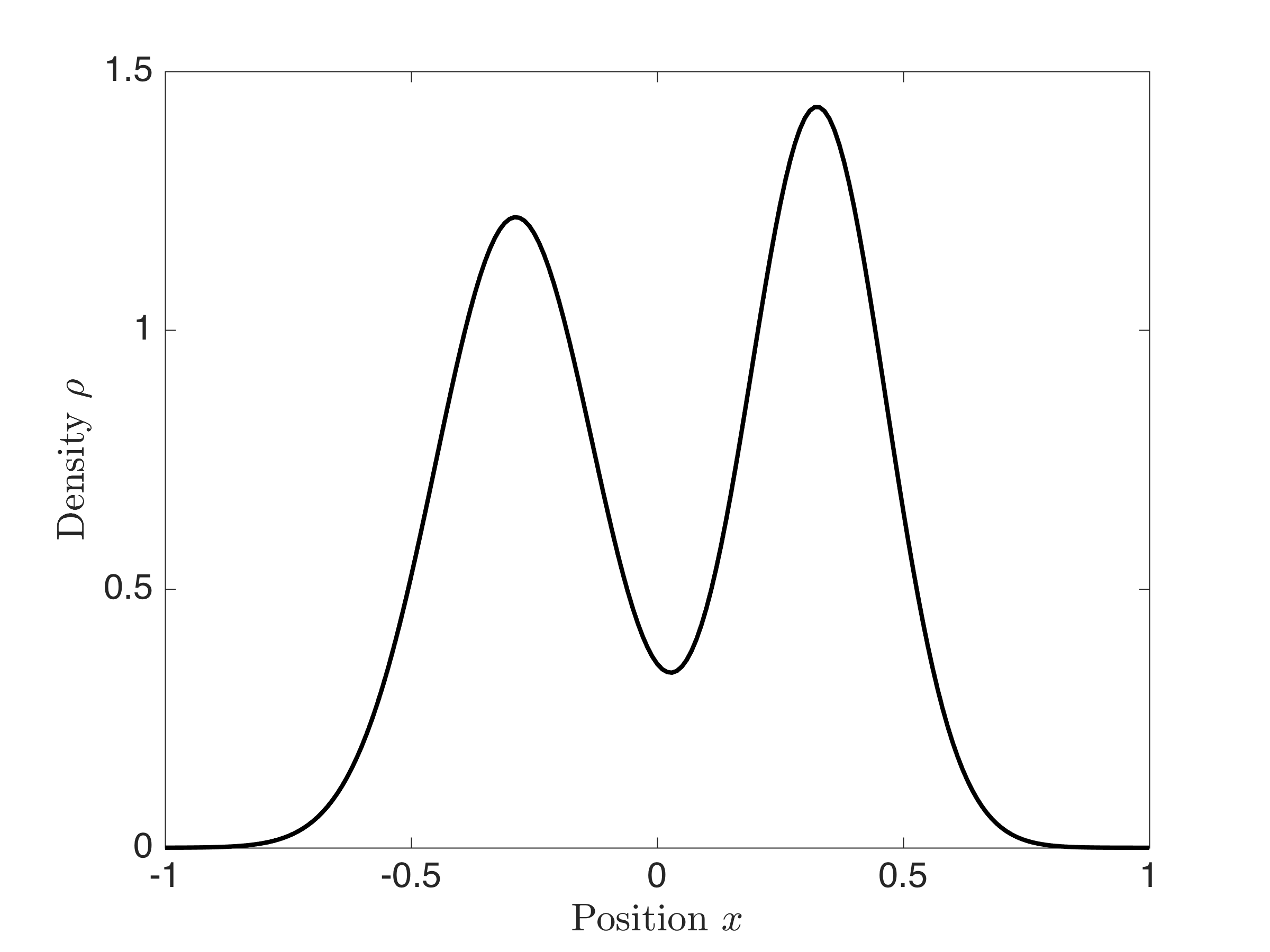}}
\subfloat[optimal transport]{\includegraphics[width=0.40\textwidth]{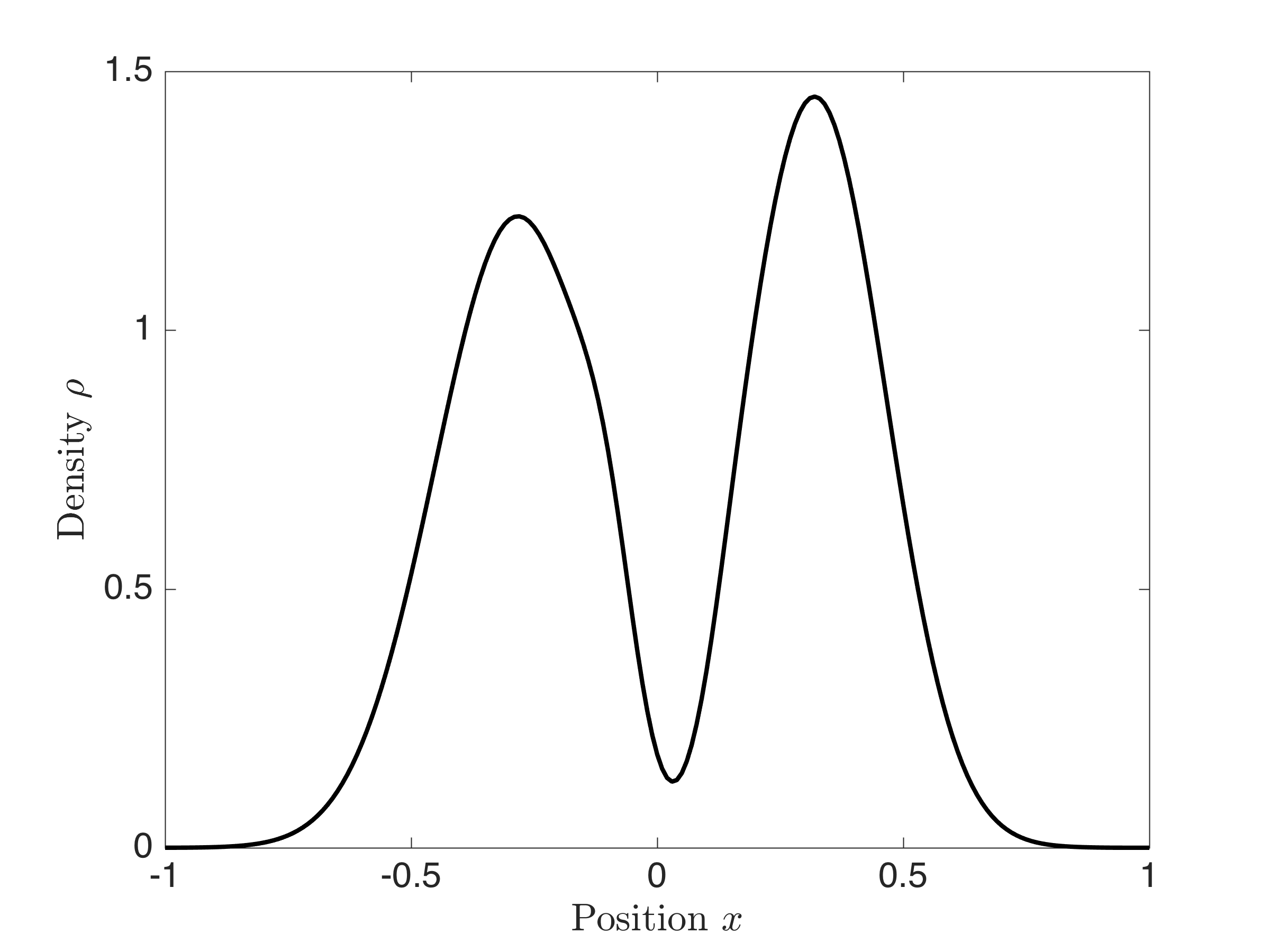}}
 \caption{Barycenters with $\lambda = (1/4, 1/4, 1/2)$}
 \label{fig:eg3barycenter1}
\end{figure}

\section{Conclusion}
In this note,  we have defined a new optimal mass transport distance for Gaussian mixture models by restricting ourselves to the submanifold of Gaussian mixture distributions. Consequently, the geodesic interpolation utilizing this metric remains on the submanifold of Gaussian mixture distributions. On the numerical side, computing this distance between two densities is equivalent to solving a linear programming problem whose number of variables grows linearly as the number of Gaussian components. This is a huge reduction in computational cost compared with traditional OMT. Finally, when the covariances of the components are small, our distance is a very good approximation of the standard OMT distance. The extension to general mixture models or structural models will be an interesting direction in the future.

\section*{Acknowledgements}
This project was supported by AFOSR grants (FA9550-15-1-0045 and FA9550-17-1-0435), grants from the National Center for Research Resources (P41-
RR-013218) and the National Institute of Biomedical Imaging and Bioengineering (P41-EB-015902), NCI grant (1U24CA18092401A1,) and NIA grant (R01 AG053991), and the Breast Cancer Research Foundation.

%\subsubsection*{Acknowledgments}
%
%Use unnumbered third level headings for the acknowledgments. All
%acknowledgments go at the end of the paper. Do not include
%acknowledgments in the anonymized submission, only in the final paper.

{
\bibliographystyle{IEEEtran}
\bibliography{./refs}
}

\end{document}